\documentclass{article}
\usepackage{amsthm, amssymb, amsmath}
\usepackage{hyperref}

\newtheorem{theorem}{Theorem}[section]

\newtheorem{lemma}[theorem]{Lemma}

\newtheorem{proposition}[theorem]{Proposition}
\newtheorem{corollary}[theorem]{Corollary}

\newcommand{\ip}[2]{\left\langle #1, #2 \right\rangle}

\newcommand{\R}{{\mathbb R}}
\newcommand{\Rp}{{\mathbb R}_+}
\newcommand{\lb}{\left(}
\newcommand{\rb}{\right)}
\newcommand{\lsb}{\left[}
\newcommand{\rsb}{\right]}
\newcommand{\p}{\partial}

\renewcommand{\O}[1]{\mathrm{O}\lb #1\rb}
\huge
\begin{document}
\title{Neck Pinching Dynamics Under Mean Curvature Flow}
\author{Zhou Gang$^{*}$, Israel Michael Sigal\thanks{Supported by NSERC under Grant NA7901.}}\maketitle
\centerline{Department of Mathematics, University of Toronto,
Toronto, Canada, M5S 2E4}
\setlength{\leftmargin}{.1in}
\setlength{\rightmargin}{.1in} 
%
%
\normalsize \vskip.1in \setcounter{page}{1}
\setlength{\leftmargin}{.1in} \setlength{\rightmargin}{.1in}
\section*{Abstract}
In this paper we study motion of surfaces of revolution under the
mean curvature flow. For an open set of initial conditions close to
cylindrical surfaces we show that the solution forms a ``neck''
which pinches in a finite time at a single point. We also obtain a
detailed description of the neck pinching process.
\section{Introduction}
In this paper we study motion of surfaces of revolution under the
mean curvature flow. The mean curvature flow of an initial
hypersurface $M_{0}\in \mathbb{R}^{d+1}$ parameterized by $\psi_{0}:
U\rightarrow M_0$ is a family of hypersurfaces $M_{t}\in
\mathbb{R}^{d+1}$ whose local parametrizations
$\psi(\cdot,t):U\rightarrow \mathbb{R}^{d+1}$ satisfy the partial
differential equation
$$
\partial_{t}\psi(z,t)=-h(\psi(z,t))
$$ where $h(y)$ is the mean curvature vector of $M_t$ at a point
$y\in M_{t}$, with the initial condition $$\psi(z,0)=\psi_{0}(z).$$

If $d\geq 2$ and an initial surface $M_{0}$ is a surface of
revolution around the axis $x=x_{d+1}$, given by a map $r=u_{0}(x)$
where $r=(\sum_{j=1}^{d}x_{j}^{2})^{ \frac{1}{2} }$, then the
surface $M_{t}$ is also a surface of revolution and, as long as it
is smooth, it is defined by the map $r=u(x,t)$ which
satisfies the partial differential equation
\begin{equation}\label{eq:MCF}
\begin{array}{lll}
\partial_{t}u&=&\frac{\partial^{2}_{x}u}{1+(\partial_{x}u)^{2}}-\frac{d-1}{u}\\
u(x,0)&=&u_{0}(x).
\end{array}
\end{equation} This equation follows from the mean curvature equation above
by a standard computation.

The initial conditions for (1) can be divided into two basic groups.
In the first group, $u_0(x)>0$ for $a<x<b$ and either
$u_0(a)=u_0(b)=0$ or $\partial_xu_0(a)=\partial_xu_0(b)=0$, for some
$-\infty<a<b<\infty$. In the second group, $u_0(x)>0\ \forall x\in
\mathbb{R}$ and $\liminf_{|x|\to\infty}u_0(x)>0$.  In the first case
we deal with compact or periodic initial surfaces and Eqn(1) is
considered on the bounded interval $[a,b]$ with the Dirichlet or
Neumann boundary conditions. In the second case, the initial surface
as well as solution surfaces
are noncompact and Equation (~\ref{eq:MCF}) should be considered on
$\mathbb{R}.$ In this paper we study the second, more difficult case
and consequently we consider Eqn \eqref{eq:MCF} on $\mathbb{R}$. Our
goal is to describe the phenomenon of collapse or neckpinching of
such surfaces.
We say $u(x,t)$ collapses at time $t^*$ if
$\|\frac{1}{u(\cdot,t)}\|_{\infty}<\infty$ for $t<t^*$ and
$\|\frac{1}{u(\cdot,t)}\|_{\infty}\rightarrow \infty$ as
$t\rightarrow t^*$.

The study of the mean curvature flow goes back at least to  the work
of Brakke ~\cite{MR485012}. The short time existence in $L^{\infty}$
was proved in ~\cite{MR485012,Huis1,MR1770903,MR1189906}. In
~\cite{Huis1,MR837523} Huisken has shown that compact convex
surfaces shrink under the mean curvature flow into a point
approaching spheres asymptotically. In some of the first works on
collapse Grayson ~\cite{MR1016434} and Ecker ~\cite{MR1139032} have
constructed rotationally symmetric barriers which can be used to
determine a class of $2-$dimensional hypersurfaces of barbell shape
which develop a singularity under the mean curvature flow before
they shrink to a point.

Huisken ~\cite{Huis2} showed that periodic rotationally symmetric,
positive mean curvature surfaces of the barbell shapes always
develop singularities in finite time  $t^*$,
and that their blow-up at a point, where the maximal curvature blows
up, converges to a cylinder of unit radius.  (No information on the
set of blow-up points was given.) These results were generalized to
higher dimensions in ~\cite{SM}.


Dziuk and Kawohl ~\cite{MR1122308} showed that periodic surfaces of
revolution of positive mean curvature which have one minimum per
period and satisfy certain monotonicity conditions, including one on
the derivative of curvature, pinch at exactly the point of minimum.

H.M.Soner and P.E.Souganidis \cite{SS} considered Equation
(~\ref{eq:MCF}) on a bounded, symmetric interval and showed that if
$u(x,t)$ is even and satisfies $x\partial_xu(x,t) \ge 0$ (i.e. $u$
has a single minimum at $x=0$), then,
along a subsequence,
\begin{equation}\label{eq:conv}
(t^*-t)^{-\frac{1}{2}}u((t^*-t)^{-\frac{1}{2}}y,t)\rightarrow\sqrt{2(d-1)},
\end{equation}
as $t\rightarrow t^*$ (a compactness result).
Smoczyk ~\cite{SMO} showed pinching of certain periodic rotationally
symmetric surfaces with the mean curvatures greater than 2 , which
are embedded in Euclidean space.

S.Altschuler, S.B.Angenent and Y.Giga ~\cite{AlAnGi} have showed
that any compact, connected, rotationally symmetric hypersurface
that pinches under the mean curvature flow  does so at finitely many
discrete points.

A collapsing solution is called of type I if the square root, $|A|$,
of the sum of squares of principal curvatures is bounded as $|A| \le
C(t^*-t)^{-\frac{1}{2}}$. Otherwise, it is called of type II (see
Huisken \cite{MR1482035}). It was conjectured that the generic
collapse is of type I. Indeed, all collapses investigated in the
papers above are of type I.

S.B.Angenent and J.J.L.Vel\'azquez ~\cite{MR1427656} have
constructed non-generic, type II solutions, first suggested by
R.Hamilton and investigated by the level-set methods of Evans and
Spruck and Chen, Giga and Goto in
~\cite{MR1770903,MR1100211,AlAnGi}, neckpinching at $x=0$ at a
prescribed time $t^*$. Their solutions have the asymptotics, as
$t\rightarrow t^*$,
%
$$ (t^*-t)^{-\frac{1}{2}}u((t^*-t)^{-\frac{1}{2}}y,t)=\sqrt{2(d-1)}
+(t^*-t)^{\frac{m}{2}-1}\frac{K}{2\sqrt{2(d-1)}}H_m(y)
+o((t^*-t)^{\frac{m}{2}-1})$$
where $K>0,\ m$ is an odd integer $ \ge 3$ and $H_m(y)$ is a
multiple of the $m$th Hermite polynomial.

Athanassenas ~\cite{MR1456315,MR1979113} has shown neckpinching of
certain class of rotationally symmetric surfaces under the volume
preserving modification of the mean curvature flow.   The latter was
also studied in Alikakos and Frere ~\cite{MR2029906}.

For other related works we refer to
~\cite{MR840401,MR1216584,WMTao,MR1666878,MR1979113}.

Most of these works rely on parabolic maximum principle going back
to Hamilon ~\cite{MR664497} and monotonicity formulae for an entropy
functional (Huisken ~\cite{Huis2}, Giga and Kohn ~\cite{GK1}).


The scaling and asymptotics in (~\ref{eq:conv}) originate in the
following key properties of (~\ref{eq:MCF}):
\begin{enumerate}
 \item (~\ref{eq:MCF}) is invariant
 with respect to the scaling transformation,
 \begin{equation}\label{rescale}
 u(x,t)\rightarrow \lambda
u(\lambda^{-1} x,\lambda^{-2} t)
\end{equation} for any constant $\lambda>0,$ i.e. if
$u(x,t)$ is a solution, then so is $\lambda u(\lambda^{-1}
x,\lambda^{-2}t).$
 \item (~\ref{eq:MCF}) has $x-$independent (cylindrical)
 solutions:
\begin{equation}\label{eq:homegeous}
 u_{cyl}=[u_{0}^{2}-2(d-1)t]^{\frac{1}{2}}.
\end{equation} These solutions collapse in finite time
 $t^* =\frac{1}{2(d-1)}u_{0}^{2}$.
\end{enumerate}

In this paper we consider Equation \eqref{eq:MCF} with initial
conditions which are positive, have, modulo small perturbations,
global minima at the origin, are slowly varying near the origin and
are even. The latter condition is of a purely technical nature and
will be addressed elsewhere. We show that for such initial
conditions the solutions collapse in a finite time and we
characterize asymptotic dynamics of the collapse. As it turns out,
the leading term is given by the expression
\begin{equation}\label{eq:decompU}
u(x,t)=\lambda(t)[(\frac{2(d-1)+b(t)\lambda^{-2}(t)x^{2}}{c(t)})^{\frac{1}{2}}+\zeta(x,t)]
\end{equation}
with the parameters $\lambda(t), \ b(t)$ and $c(t)$ satisfying the
estimates
\begin{equation}\label{eq:para}
\begin{array}{lll}
\lambda(t)&=&(t^*-t)^{\frac{1}{2}}(1+o(1));\\
& &\\
b(t)&=&-\frac{d-1}{ln|t^*-t|}(1+O(\frac{1}{|ln|t^*-t|}|^{3/4}));\\
& &\\
c(t)&=&1+\frac{1}{ln|t^*-t|}(1+O(\frac{1}{ln|t^*-t|})).
\end{array}
\end{equation}
Here
$\lambda_{0}=\frac{1}{\sqrt{2\varsigma_{0}+\frac{\varepsilon_{0}}{d-1}}}$
with $\varsigma_0,\ \varepsilon_{0}>0$ depending on the initial
datum and $o(1)$ is in $t^*-t$. Moreover, we estimate the remainder
$\zeta(x,t)$ as
\begin{equation}\label{eq:EstRemainder}
\sum\limits_{m+n=3,n\leq2} \|\langle
\lambda^{-1}(t)x\rangle^{-m}\partial^n_x\zeta(x,t)\|_{\infty}\leq c
b^{2}(t)
\end{equation} for some constant $c.$

To give more precise formulation of results we introduce some
notation. Let $L^{\infty}$ denote the space $L^{\infty}(\mathbb{R})$
with the standard norm $\|u\|_{\infty}=\sup_{x}|u(x)|$. To formulate
our main result we define the spaces $L_{m,n}^{\infty}$ with the
norm
$$\|u\|_{m,n}=\|\langle x\rangle^{-m}\partial_{x}^{n}u(x)|\|_{\infty}$$ and define
the function $g(x,b),\ b>0,$ as
\begin{equation}\label{eq:Lower}
g(x,b):=\left\{
\begin{array}{lll}
\frac{9}{10}\sqrt{2(d-1)}\ \text{if}\ b x^{2}< 20(d-1)\\
4\sqrt{d-1}\ \ \ \ \ \ \text{if}\ \ b x^{2}\geq 20(d-1)\ .
\end{array}
\right.
\end{equation} We will also deal, without specifying it, with weak
solutions of Equation (~\ref{eq:MCF}) in some appropriate sense (see
the next section for more precise formulation). These solutions can
be shown to be classical for $t>0.$ The following is the main result
of our paper.
\begin{theorem}\label{maintheorem}
Assume the initial datum $u_{0}(x)$ in (~\ref{eq:MCF}) is even and
satisfy for $(m,n)=(3,0),$ $(\frac{11}{10},0)$, $(1,2)$ and $(2,1)$
the estimates
\begin{equation}\label{eq:INI2}
\begin{array}{ccc}
&
&\|u_{0}(x)-(\frac{2(d-1)+\varepsilon_{0}x^{2}}{2\varsigma_{0}})^{\frac{1}{2}}\|_{m,n}\leq
C \varepsilon_{0}^{\frac{m+n+1}{2}},\\
& & \\
& &u_{0}(x)\geq
\frac{1}{\sqrt{2\varsigma_{0}+\frac{\varepsilon_{0}}{d-1}}}
g(\sqrt{2\varsigma_{0}+\frac{\varepsilon_{0}}{d-1}}x,
\frac{\varepsilon_{0}}{2\varsigma_{0}}),
\end{array}
\end{equation} $\langle x\rangle^{-1}u_{0}\in  L^{\infty},$ $\partial_{x}u_{0}\in L^{\infty}$,
$|\partial_{x}u_{0}u^{- \frac{1}{2} }_{0}|\leq \kappa_{0}
\varepsilon^{ \frac{1}{2} }_{0},$ $|\partial_{x}^{n}u_{0}|\leq
\kappa_{0}\varepsilon_{0}^{n/2}$, $n=2,3,4,$ for some $C,\
\kappa_{0}\geq 2$, and $ \frac{1}{2} \leq \varsigma_{0}\leq 2$.
There exists a constant $\delta$ such that if $\varepsilon_{0}\leq
\delta$, then
\begin{enumerate}
\item[(i)] there exists a finite time $t^*$ such that $\|\frac{1}{u(\cdot,t)}\|_{\infty}<\infty$
for $t<t^*$ and $\lim_{t\rightarrow
t^*}\|\frac{1}{u(\cdot,t)}\|_{\infty}\rightarrow \infty;$
\item[(ii)] there exist $C^{1}$ functions $\zeta(x,t)$, $\lambda(t),\ c(t)$ and $b(t)$
such that (~\ref{eq:decompU}) and (~\ref{eq:EstRemainder}) hold;
\item[(iii)] the parameters $\lambda(t), \ b(t)$ and $c(t)$ satisfy the estimates
(~\ref{eq:para});
\item[(iv)] if
$u_{0}\partial_{x}^{2}u_{0}\geq -1$ then there exists a function
$u_*(x)>0$ such that $u(x,t)\geq u_*(x)$ for
$\mathbb{R}\backslash\{0\}$ and $t\leq t^*$. Moreover, if the mean
curvature of the initial surface is non-negative, i.e., if
$\frac{\partial_{x}^{2}u_{0}}{1+(\partial_{x}u_{0})^{2}}-\frac{d-1}{u_0}\leq
0$, then for any $x$, $\lim_{t\rightarrow t^{*}}u(x,t)$ exists and
is positive $\forall x\neq 0$.
\end{enumerate}
\end{theorem}

Thus our main new results are
\begin{enumerate}
\item[1)] Proof of neckpinching and neckpinching asymptotics at a single
given point for a new open set of initial conditions, which includes
in particular surfaces whose mean curvature changes sign and which
might have many necks,
\item[2)] Determination of the subleading term in the
asymptotic and estimation of the remainder.
\end{enumerate}
Remarks
\begin{enumerate}
\item[1)] A result similar to (iv) but for a different set of
initial conditions (see above) was proven in H.M.Soner and
P.E.Souganidis \cite{SS};
\item[2)]One can compute more precise asymptotics of the parameters
$\lambda(t), \ b(t)$ and $c(t)$;
\item[3)]It is not hard to show using \eqref{eq:decompU}-\eqref{eq:EstRemainder}
that the collapse in our case is of type I.
\end{enumerate}

The previous result closest to our result is that by Angenent and
Knopf ~\cite{MR2092903,SK2} on the neckpinching for the Ricci flow
of $SO(n+1)-$ invariant metrics on $S^{n+1}.$

Our techniques are different from those in the papers mentioned
above. They rely to much lesser degree on the maximum principle and
they do not use entropy monotonicity formulae. Our main point is
that we do not fix the time-dependent scale in the self-similarity
(collapse) variables but let its behaviour, as well as behaviour of
other parameters ($b$ and $c$), be determined by the original
equation. Then we use a nonlinear Lyapunov-Schmidt decomposition
(the modulation method) and the method of majorants together with
powerful linear estimates. We expect that our techniques can be
extended to non-axisymmetric surfaces and to Ricci flows.

This paper is organized as follows. In Section ~\ref{SEC:localWell}
we prove the local well-posedness of Equation (~\ref{eq:MCF}) in the
space $\langle x\rangle L^{\infty}$ which is used in this paper. In
Sections \ref{Sec:CollapseVar}-\ref{Section:Reparam} we present some
preliminary derivations and some motivations for our analysis. In
Section \ref{SEC:ApriEst}, we formulate a priori bounds on solutions
to \eqref{eq:MCF}. In Section \ref{SecMain} we use these bounds and
a lower bound proved in Section ~\ref{SEC:LowerBound} to prove our
main result, Theorem \ref{maintheorem}. A priori bounds of Section
\ref{SEC:ApriEst} are proved in Sections
\ref{SEC:EstB}-~\ref{SEC:estM12}.

For any functions $A$ and $B$ we use the notation $A\lesssim B$ to
signify that there is a universal constant $c$ such that $A\leq cB.$
\section*{Acknowledgement} The authors are grateful to Steven Dejak and
Shuangcai Wang for many fruitful discussions.
\section{Local Well-posedness of
(~\ref{eq:MCF})}\label{SEC:localWell} In this section we prove the
local well posedness of (~\ref{eq:MCF}) in the space adapted to our
needs. The result below is standard (cf
~\cite{MR1770903,MR1189906,MR485012}.)
\begin{theorem}\label{THM:WellPose}
If $u_{0}(x)\in \langle x\rangle  L^{\infty}$ and $u_{0}(x)\geq
C_{0}$ for some $C_{0}>0$ and $\partial_{x}^{n}u_{0}\in L^{\infty},\
n=1,2,3,4$ then there exists a time $T\equiv T(C_{0},u_{0})$ such
that for any time $0\leq t\leq T$, (~\ref{eq:MCF}) has a unique
solution $u(\cdot,t)\in \langle x\rangle L^{\infty}$ with
$u(x,0)=u_{0}(x)$, $u(x,t)\geq \frac{C_{0}}{2}$ and
$\partial_{x}^{n}u(\cdot,t) \in L^{\infty}, n=1,2,3,4.$ Moreover, if
$t_*$ is the supremum of such $T(\kappa_{0},u_{0})$ then either
$t^{*}=\infty$ or $\|\frac{1}{u(\cdot,t)}\|_{\infty}\rightarrow
\infty$ as $t\rightarrow t_*.$
\end{theorem}
\begin{proof} First we consider the equation
\begin{equation}\label{eq:u1xt}
\begin{array}{lll}
\partial_{t}u_{1}&=&g_{1}(\partial_{x}u_{1})\partial_{x}^{2}u_{1}-(d-1)g_{2}(u_{1})u_{1}\\
& &\\
u_{1}(x,0)&=& u_{0}(x)
\end{array}
\end{equation}
where $g_{1}$ and $g_{2}$ are strictly positive and smooth functions satisfying the conditions
$$
g_{1}(s):=\left\{
\begin{array}{lll}
\frac{1}{1+s^{2}}\ \text{if}\ s\leq 10\|\partial_{x}u_{0}\|_{\infty},\\
1\ \ \ \ \ \text{if}\ s\geq 20 \|\partial_{x}u_{0}\|_{\infty},
\end{array}
\right.
$$ and $$g_{2}(s):=\left\{
\begin{array}{lll}
\frac{1}{s^{2}}\ \text{if}\ s\geq \frac{1}{10}C_{0},\\
1\ \ \  \text{if}\ s\leq \frac{1}{20} C_{0}.
\end{array}
\right.$$

By standard results (see ~\cite{LSU}) there exists a time $T>0$ such
that (~\ref{eq:u1xt}) has a unique solution $u_{1}(x,t)$ in the time
interval $t\in [0,T]$ such that $u_{1}\geq \frac{1}{2}C_{0}$ and
$\partial_{x}^{n}u_{1}(\cdot,t)\in L^{\infty},$ $n=1,2,3,4,$ and
$\|\partial_{x}u_{1}(\cdot,t)\|_{\infty}\leq
2\|\partial_{x}u_{0}\|_{\infty}.$ Moreover by the definition of
$g_{1}$ and $g_{2}$ we have that for this solution
$$g_{1}(\partial_{x}u_{1})=\frac{1}{1+(\partial_{x}u_{1})^{2}},\
g_{2}(u_{1})=\frac{1}{u_{1}^{2}},$$ in the time interval $t\in
[0,T].$ Thus $u(x,t)=u_{1}(x,t)$ is a solution to (~\ref{eq:MCF})
and satisfies all the conditions of the theorem.
\end{proof}
\section{Collapse Variables and Almost
Solutions}\label{Sec:CollapseVar} In this section we pass from the
original variables $x$ and $t$ to the collapse variables
$y:=\lambda^{-1}(t) (x-x_{0}(t))$ and $\tau:=\int_{0}^{t}
\lambda^{-2}(s) \,ds$.  The point here is that we do not fix
$\lambda(t)$ and $x_0(t)$ but consider them as free parameters to be
found from the evolution of \eqref{eq:MCF}. Suppose $u(x,t)$ is a
solution to \eqref{eq:MCF} with an initial condition $u_{0}(x)$,
which has a minima at $x=0$ and is even with respect to $x=0$. We
define the new unknown function
\begin{equation}\label{eq:definev}
v(y,\tau):=\lambda^{-1}(t)u(x,t)
\end{equation} with $y:=\lambda^{-1}(t)x$ and $\tau:= \int_{0}^{t}\lambda^{-2}(s)ds.$
The function $v$ satisfies the equation
\begin{equation}\label{eqn:BVNLH}
\partial_{\tau}v=(\frac{1}{1+(\partial_{y}v)^{2}}\partial_y^2-a
y\partial_y+a) v-\frac{d-1}{v},
\end{equation}
where $a:=-\lambda\partial_t\lambda$.  The initial condition is
$v(y,0)=\lambda^{-1}_{0} u_{0}( \lambda_{0}y)$, where $\lambda_0$ is
the initial condition for the scaling parameter $\lambda$.

If $\lambda_{0}=1$, then the initial conditions for $u$ given in
Theorem ~\ref{maintheorem} implies that there exists a constant
$\delta$ such that the initial condition $v_{0}(y)$ is even and
satisfy for $(m,n)=(3,0),$ $(\frac{11}{10},0)$, $(1,2)$ and $(2,1)$
the estimates
\begin{equation}
\begin{array}{ccc}
&
&\|v_{0}(y)-(\frac{2(d-1)+\varepsilon_{0}y^{2}}{1-\frac{1}{d-1}\varepsilon_{0}})^{\frac{1}{2}}\|_{m,n}\leq
C \varepsilon_{0}^{\frac{m+n+1}{2}},\\
& & \\
& &v_{0}(y)\geq g(y, \varepsilon_{0}),
\end{array}
\end{equation} $\langle y\rangle^{-1}v_{0}\in  L^{\infty},$ $\partial_{y}v_{0}\in L^{\infty}$,
$|\partial_{y}v_{0}v^{- \frac{1}{2} }_{0}|\leq \kappa_{0}
\varepsilon^{ \frac{1}{2} }_{0},$ $|\partial_{y}^{n}v_{0}|\leq
\kappa_{0}\varepsilon_{0}^{n/2}$, $n=2,3,4,$ for some $C,\
\kappa_{0}\geq 2$, $ \frac{1}{2} \leq \varsigma_{0}\leq 2$ and
$\varepsilon_{0}\leq \delta$.

If the parameter $a$ is a constant, then \eqref{eqn:BVNLH} has the
following cylindrical, static (i.e. $y$ and $\tau$-independent)
solution
\begin{equation}
v_a:=( \frac{d-1}{ a})^{\frac{1}{2}}.
\end{equation}
In the original variables $t$ and $x$, this family of solutions
corresponds to the homogeneous solution \eqref{eq:homegeous} of
(~\ref{eq:MCF}) with the parabolic scaling $\lambda^{2}=2 a(T-t)$,
where the collapse time, $T:= \frac{u_{0}^{2}}{2(d-1)}$, is
determined by $u_0$, the initial value of the homogeneous solution
$u_{hom}(t)$.

If the parameter $a$ is $\tau-$dependent but $|a_\tau|$ is small,
then the above solutions are good approximations to the exact
solutions. A larger family of approximate solution is obtained by
solving the equation $ a y v_y-a v+\frac{d-1}{v}=0$, which is
derived from \eqref{eqn:BVNLH} by neglecting the $\tau$ derivative
and second order derivative in $y$ (adiabatic, slowly varying
approximation). This equation has the general solution
\begin{equation}
v_{b c}:=(\frac{2(d-1) + b y^2}{c})^{\frac{1}{2}}
\end{equation}
for all $b\in\mathbb{R}$ and for $c=2a$. In what follows we take
$b\ge 0$ so that $v_{b c}$ is smooth.  Note that $v_{0,2a}=v_a$.
Since $v_{bc},\ c=2a$, is not an exact solution to
(~\ref{eqn:BVNLH}) we should leave the parameter $c$ free, to be
determined by the best overall approximation. Jumping ahead, it
turns out that a convenient choice of $c$ is $c=a+\frac{1}{2}.$ Thus
we introduce $V_{ab}(y):=v_{b,a+\frac{1}{2}}(y)=(\frac{2(d-1)+b
y^{2}}{a+\frac{1}{2}})^{\frac{1}{2}}.$ Note that this function is
not a solution to the equation $ay\partial_{y}v-av+\frac{d-1}{v}=0.$

\section{``Gauge" Transform}
\label{sec:Gauge}
In order to convert the non-self-adjoint linear part of Equation
\eqref{eqn:BVNLH} into a more tractable self-adjoint one we
perform a gauge transform. Let
\begin{equation}\label{eq:definew}
w(y,\tau):=\exp{-\frac{a}{4}y^{2}}v(y,\tau).
\end{equation}
Then $w$ satisfies the equation
\begin{equation} \label{eqn:w}
\partial_{\tau}w=(
\frac{1}{1+(\partial_{y}v)^{2}}\partial_y^2-\frac{1}{4}\omega^2
y^2+\frac{3}{2}a ) w-\exp{-\frac{a}{2}y^2}\frac{d-1}{w},
\end{equation}
where $\omega^2=a^2+a_\tau$.  The approximate solution $v_{bc}$ to
\eqref{eqn:BVNLH} transforms to $v_{a b c}$ where $v_{a b c}:=v_{b
c}\exp{-\frac{a}{4}y^{2}}$, explicitly
\begin{equation}
v_{a b c}:=(\frac{2(d-1)+b y^2}{
c})^{\frac{1}{2}}\exp{-\frac{a}{4}y^{2}}.
\end{equation} As was permitted above we will choose $c=a+\frac{1}{2}.$

Note that the linear part of Equation (~\ref{eqn:w}) is self-adjoint
in the space $L^{2}(\mathbb{R},dy).$ Hence it is natural to consider
the linear part of Equation (~\ref{eqn:BVNLH}) in the space
$L^{2}(\mathbb{R}, e^{-\frac{a y^{2}}{2}}dy).$
\section{Reparametrization of Solutions}
\label{Section:Reparam}
In this section we split solutions to Equation \eqref{eqn:BVNLH}
into the leading term - the almost solution $V_{a b}$ - and a
fluctuation $\eta$ around it.  More precisely, we would like to
parametrize a solution by a point on the manifold $M_{as}:=\{V_{a b
}\, |\, a,b\in \Rp, b \leq \epsilon\}$ of almost solutions and the
fluctuation orthogonal to this manifold (large slow moving and small
fast moving parts of the solution). We equip $M_{as}$ with the
Riemannian metric
\begin{equation}\label{eq:Rie}
\langle \eta,\eta^{'}\rangle:=\int \eta\eta^{'}e^{-\frac{a
y^{2}}{2}}d y.
\end{equation}
For technical reasons, it is more convenient to require the
fluctuation to be almost orthogonal to the manifold $M_{as}$. More
precisely, we require $\eta$ to be orthogonal to the vectors $1$ and
$(1-a y^2)$ which are almost tangent vectors to the above manifold,
provided $b$ is sufficiently small. Note that $\eta$ is already
orthogonal to $\sqrt{a} y $ since our initial conditions, and
therefore, the solutions are even in $x$.


The next result will give a convenient reparametrization of the
initial condition $v_{0}(y):=\lambda_{0}^{-1}u_{0}(\lambda_{0}y).$
Recall the definition $V_{a
b}:=v_{bc}|_{c=a+\frac{1}{2}}=(\frac{2(d-1)+by^{2}}{a+\frac{1}{2}})^{\frac{1}{2}}$.
We define a neighborhood:
\begin{equation*}
U_{\epsilon_0}:=\{v\in \langle y\rangle^{3}L^\infty(\R)\ |\ \|
v-V_{ab}\|_{3,0}\ll b\ \mbox{for some}\ a\in[1/4,1],\
b\in(0,\epsilon_0]\ \}.
\end{equation*}

\begin{proposition}\label{Prop:Splitting}
There exist an $\epsilon_{0}>0$ and a unique $C^1$ functional
$g:U_{\epsilon_0}\rightarrow \mathbb{R}^{+}\times \mathbb{R}^{+}$,
such that any function $v\in U_{\epsilon_0}$ can be uniquely written
in the form
\begin{equation} \label{eqn:split}
v =V_{g(v)} + \eta,
\end{equation}
with $\eta\perp 1,\ 1-a y^{2}$ in $L^2(\R, e^{-\frac{a
y^{2}}{2}}dy)$, $(a,b)=g(v)$. Moreover, if $(a_0, b_0)\in
[\frac{1}{4},1]\times (0,\epsilon_0]$ and
$\|v-V_{a_{0},b_{0}}\|_{3,0}\ll b_{0}$, then
\begin{equation}
\label{eqn:28a} |g(v)-(a_{0},b_{0})|\lesssim
\|v-V_{a_{0}b_{0}}\|_{3,0}.
\end{equation}
\end{proposition}
\begin{proof}
Let $X:= \langle y\rangle^3 L^\infty$ with the corresponding norm.
The orthogonality conditions on the fluctuation can be written as
$G(\mu,v)=0$, where $\mu=(a,b)$ and $G:\mathbb{R}^{+}\times
\mathbb{R}^{+}\times X \rightarrow \R^2$ is defined as (we use the
Riemannian metric (~\ref{eq:Rie}))
\begin{equation*}
G(\mu, v):=\lb
\begin{array}{ccc} \langle V_{\mu}-v,1\rangle\\
\langle V_{\mu}-v,1-a y^{2}\rangle
\end{array} \rb.
\end{equation*}
Using the implicit function theorem we will prove that for any
$\mu_0:=(a_0, b_0)\in [\frac{1}{4},1]\times (0,\epsilon_{0}]$ there
exists a unique $C^1$ function $g:U_{\mu_{0}}\rightarrow
\mathbb{R}^{+}\times \mathbb{R}^{+}$ defined in a neighborhood
$U_{\mu_0}\subset X$ of $V_{\mu_0}$ such that $G(g(v),v)=0$ for all
$v\in U_{\mu_0}$.

Note first that the mapping $G$ is $C^1$ and $G(\mu_0, V_{\mu_0})=0$
for all $\mu_0$.  We claim that the linear map $\p_\mu G(\mu_0,
V_{\mu_0})$ is invertible.  Indeed, let $B_{\epsilon}(V_{\mu_0})$
and $B_\delta(\mu_0)$ be the balls in $X$ and $\mathbb{R}^2$ around
$V_{\mu_0}$ and $\mu_0$ and of the radii $\epsilon$ and $\delta$,
respectively. We compute
\begin{equation}\label{linearization}
\p_\mu G(\mu, v)=A_{1}(\mu)+A_{2}(\mu,v)
\end{equation}
where
$$
A_{1}(\mu):=\lb
\begin{array}{cc}
\ip{\p_a V_{
\mu}}{ 1} & \ip{\p_b V_{\mu}}{ 1}  \\
\ip{\p_a V_{\mu}}{1-a y^2}& \ip{\p_b V_{\mu}}{1-a y^2}
\end{array}
\rb
$$ and
\begin{equation*}
A_{2}(\mu,v):= -\frac{1}{4}\lb\begin{array}{cc}
\langle V_{\mu}-v, y^{2}\rangle& 0 \\
\ip{V_{\mu}-v}{\lb 1-a y^2\rb y^2 } & 0
\end{array}\rb.
\end{equation*}
For $b>0$ and small, we expand the matrix $A_{1}$ in $b$ to get
$A_{1}=G_{1}+O(b)$, where the matrices $G_{1}$ is defined as
$$G_{1}:=\left(
\begin{array}{ccc}
-\frac{1}{2}(a+\frac{1}{2})^{-3/2}\langle 1,1\rangle &
\frac{1}{2(a+\frac{1}{2})}\langle y^{2},1\rangle\\
-\frac{1}{2}(a+\frac{1}{2})^{-3/2}\langle 1, 1-ay^{2}\rangle &
\frac{1}{2(a+\frac{1}{2})}\langle y^{2},1-ay^{2}\rangle
\end{array}
\right).$$ Obviously the matrices $G_{1}$ has uniformly (in $a\in
[\frac{1}{4},1]$) bounded inverses. Furthermore, by the Schwarz
inequality
$$\|A_{2}(\mu,v)\|\lesssim \| v-V_{ab}\|_X.$$
Therefore there exist $\epsilon_0$ and $\epsilon_1$ s.t. the matrix
$\p_\mu G(\mu,v)$ has a uniformly bounded inverse for any $v\in
B_{\epsilon_1}(V_{\mu})$ and
$\mu\in[\frac{1}{4},1]\times(0,\epsilon_0]$. Hence by the implicit
function theorem, the equation $G(\mu,v)=0$ has a unique solution
$\mu=g(v)$ on a neighborhood of every $V_\mu$,
$\mu\in[\frac{1}{4},1]\times(0,\epsilon_0]$, which is $C^1$ in $v$.
Our next goal is to determine these neighborhoods.

To determine a domain of the function $\mu=g(v)$, we examine closely
a proof of the implicit function theorem. Proceeding in a standard
way, we expand the function $G(\mu,v)$ in $\mu$ around $\mu_0$:
\begin{equation*}
G(\mu,v)=G(\mu_0,v)+\p_\mu G(\mu_0,v)(\mu-\mu_0)+R(\mu,v),
\end{equation*}
where $R(\mu,v)=\O{|\mu-\mu_0|^2}$ uniformly in $v\in X$. Here
$|\mu|^2=|a|^2+|b|^2$ for $\mu=(a,b)$.  Inserting this into the
equation $G(\mu,v)=0$ and inverting the matrix $\p_\mu G(\mu_0,v)$,
we arrive at the fixed point problem $\alpha=\Phi_v(\alpha)$, where
$\alpha:=\mu-\mu_0$ and $\Phi_v(\alpha):=- \p_\mu G(\mu_0,v)^{-1}
[G(\mu_0,v)+R(\mu,v)]$.  By the above estimates there exists an
$\epsilon_1$ such that the matrix $\p_\mu G(\mu_0,v)^{-1}$ is
bounded uniformly in $v\in B_{\epsilon_1}(V_{\mu_0})$.  Hence we
obtain from the remainder estimate above that
\begin{equation}
|\Phi_v(\alpha)|\lesssim|G(\mu_0,v)|+|\alpha|^2.
\label{eqn:SplittingSharp}
\end{equation}
Furthermore, using that $\p_\alpha \Phi_v(\alpha)= - \p_\mu
G(\mu_0,v)^{-1} [G(\mu,v)- G(\mu_0,v)+R(\mu,v)]$ we obtain that
there exist $\epsilon \leq \epsilon_1$ and $\delta$ such that
$\|\p_\alpha\Phi_v(\alpha)\|\le\frac{1}{2}$ for all $v\in
B_{\epsilon}(V_{\mu_0})$ and $\alpha\in B_\delta(0)$. Let $\mu_0 =
(a_0, b_0).$ Pick $\epsilon$ and $\delta$ so that $\epsilon
\ll\delta\ll \min(b_0, \epsilon_1) \ll 1$.  Then, for all $v\in
B_{\epsilon}(V_{\mu_0})$, $\Phi_v$ is a contraction on the ball
$B_\delta(0)$ and consequently has a unique fixed point in this
ball.  This gives a $C^1$ function $\mu=g(v)$ on
$B_{\epsilon}(V_{\mu_0})$ satisfying $|\mu-\mu_0|\le\delta$. An
important point here is that since $\epsilon\ll b_0$ we have that
$b>0$ for all $V_{ab}\in B_{\epsilon}(V_{\mu_0})$ (we use here that
$|b' -b| \le \frac{1}{c}\|\langle
y\rangle^{-3}(V_{ab'}-V_{ab})\|_\infty$). Now, clearly, the balls
$B_{\epsilon}(V_{\mu_0})$ with
$\mu_0\in[\frac{1}{4},1]\times[0,\epsilon_0]$ cover the
neighbourhood $U_{\epsilon_0}$.  Hence, the map $g$ is defined on
$U_{\epsilon_0}$ and is unique, which implies the first part of the
proposition.

Now we prove the second part of the proposition.  The definition of
the function $G(\mu,v)$ implies $G(\mu_0,v)=G(\mu_0,v-V_{\mu_0})$
and
\begin{equation}
|G(\mu_0,v)|\lesssim \|\langle y\rangle^{-3}(v-V_{\mu_0})\|_\infty.
\label{eqn:28aA}
\end{equation}
This inequality together with the estimate
\eqref{eqn:SplittingSharp} and the fixed point equation
$\alpha=\Phi_v(\alpha)$, where $\alpha=\mu-\mu_0$ and $\mu=g(v)$,
implies $|\alpha|\lesssim \| \langle
y\rangle^{-3}(v-V_{\mu_0})\|_\infty+|\alpha|^2$ which, in turn,
yields \eqref{eqn:28a}.
\end{proof}
\begin{proposition}
\label{Prop:SplittingIC} In the notation of Proposition
\ref{Prop:Splitting}, if $\|v-V_{\mu_{0}}\|_{m,n}\lesssim
b_{0}^{\frac{m+n+1}{2}}$ where $b_{0}>0$ is small and $(m,n)=(3,0),\
(\frac{11}{10},0)$, $ (2,1)$, $(1,2)$, then
\begin{equation}\label{eq:vv0}
|g(v)-\mu_{0}|\lesssim \|v-V_{\mu_{0}}\|_{3,0};
\end{equation}
\begin{equation}
\|v-V_{g(v)})\|_{3,0}\lesssim \|v-V_{\mu_{0}}\|_{3,0};
\label{Ineq:IC}
\end{equation}
\begin{equation}\label{eq:vwithoutweight}
\|v-V_{g(v)}\|_{m',n'}\lesssim b_{0}^{\frac{m'+n'+1}{2}}
\end{equation} with $(m',n')=(\frac{11}{10},0), \ (1,2)$, $(2,1).$
\end{proposition}
\begin{proof}
Equation (~\ref{eqn:28a}) implies (~\ref{eq:vv0}) with
$\mu_{0}=(a_{0},b_{0})$. Moreover we observe
$$
\begin{array}{lll}
\|v-V_{g(v)}\|_{3,0}&\leq &\|v-V_{\mu_{0}}\|_{3,0}+\|V_{g(v)}-V_{\mu_{0}}\|_{3,0}\\
&\lesssim& \|v-V_{\mu_{0}}\|_{3,0}+|\mu_{0}-g(v)|\\
&\lesssim & \|v-V_{\mu_{0}}\|_{3,0}
\end{array}
$$ which is (~\ref{Ineq:IC}).

For Equation (~\ref{eq:vwithoutweight}) we only prove the case
$(m',n')=(\frac{11}{10},0),$ the other cases are proved similarly.
We write
$$\|v-V_{g(v)}\|_{\frac{11}{10},0}\leq \|v-V_{\mu_{0}}\|_{\frac{11}{10},0}+\|V_{g(v)}-V_{\mu_{0}}\|_{\frac{11}{10},0}.$$
By the definition of $V_{a,b}$ we have
$$\|V_{g(v)}-V_{\mu_{0}}\|_{\frac{11}{10},0}\lesssim
|g(v)-\mu_{0}|b_{0}^{-\frac{9}{20}}.$$ This together with
(~\ref{eq:vv0}) implies
$$\|V_{g(v)}-V_{\mu_{0}}\|_{\frac{11}{10},0}\lesssim b_{0}^{\frac{21}{10}}.$$ Using
$\|v-V_{\mu_{0}}\|_{\frac{11}{10},0}\lesssim b_{0}^{\frac{21}{10}}$
we complete the proof of (~\ref{eq:vwithoutweight}) for
$(m',n')=(\frac{11}{10},0)$.
\end{proof}

Now we establish a reparametrization of solution $u(x,t)$ on small
time intervals. In Section \ref{SecMain} we convert this result into
a global reparametrization. In the rest of the section it is
convenient to work with the original time $t$,  instead of rescaled
time $\tau$.
We denote $I_{t_0, \delta}:= [t_0, t_0 +\delta]$ and define for any
time $t_{0}$ and constant $\delta>0$ two sets:
$$\mathcal{A}_{t_0, \delta}:=
C^{1}(I_{t_0, \delta},[\frac{1}{4},1])\ \mbox{and}\
\mathcal{B}_{t_0, \delta, \epsilon_0}:=C^{1}(I_{t_0,
\delta},(0,\epsilon_{0}])$$ where, recall, the constant
$\epsilon_{0}$ is the same as in Proposition ~\ref{Prop:Splitting}.

Denote $u_{\lambda}(y, t) := \lambda^{-1}(t)u(\lambda(t) y,t)$.
Suppose $u(\cdot,t)$ is a function such that for some
$\lambda_{0}>0$
\begin{equation}\label{eq:init2}
\sup_{t \in I_{t_0, \delta}}b^{-1}(t)\|
u_{\lambda}(\cdot,t)-V_{a(t),b(t)}\|_{3,0}\ll 1
\end{equation}
for some  $a\in\mathcal{A}_{t_0, \delta}$, $b\in\mathcal{B}_{t_0,
\delta, \epsilon_0}$, and $\lambda(t)$ satisfying
$\lambda(t_{0})=\lambda_{0}\ \mbox{and}\
-\lambda(t){\partial_t}\lambda(t)=a(t)$. We define the set
$$\mathcal{U}_{t_0,\delta, \epsilon_0, \lambda_0}:=\{u \in C^1(I_{t_0, \delta},
\langle y\rangle^3L^\infty)\ |\ (~\ref{eq:init2})\ \text{holds for
some}\ a(t), b(t) \}.$$
\begin{proposition}\label{Prop:LMSplitting2}
Suppose $u\in \mathcal{U}_{t_0,\delta, \epsilon_0, \lambda_0}$ and
$\lambda_{0}^{2}\delta \ll 1$. Then there exists a unique $C^1$ map
$g_\#:\mathcal{U}_{t_0,\delta, \epsilon_0, \lambda_0}\rightarrow
\mathcal{A}_{t_0, \delta}\times \mathcal{B}_{t_0, \delta,
\epsilon_0}$, such that for $t\in I_{t_0, \delta},$ $u(\cdot,t)$ can
be uniquely represented in the form
\begin{equation} \label{eqn:splitting2}
u_{\lambda}(y, t) =V_{g_\#(u)(t)}(y) + \phi(y,\tau(t)),
\end{equation}
with $\tau(t):=\int_{0}^{t}\lambda^{-2}(t)dt$,
$(a(t),b(t))=g_\#(u)(t)$ and
\begin{equation} \label{eqn:splitting2conditions}
\phi(\cdot,\tau(t))\perp 1,\ a(t) y^{2}-1\ \mbox{in}\
L^2(\R,e^{-\frac{a(t)}{2} y^2}dy),\ \lambda(t_{0})=\lambda_{0}\
\mbox{and}\
-\lambda(t){\partial_t}\lambda(t)=a(t).
\end{equation}
\end{proposition}
\begin{proof}  Recall the definition $X:= \langle y\rangle^3 L^\infty$
with the corresponding norm. For any function $a\in
\mathcal{A}_{t_0, \delta},$ we define a function
$$\lambda(a,t):=(\lambda_{0}^{2}-2\int_{t_0}^{t}a(s)ds)^{\frac{1}{2}}.$$ Let $\lambda(a)(t):=\lambda(a,t)$.
Define the $C^1$ map $G_\#:
C^{1}(I_{t_0, \delta},\mathbb{R}^+)\ \times C^{1}(I_{t_0,
\delta},\mathbb{R}^+)\times C^{1}(I_{t_0, \delta},X)\rightarrow
C^{1}(I_{t_0, \delta},\mathbb{R})\ \times C^{1}(I_{t_0,
\delta},\mathbb{R})$ as $$G_\#(\mu,u)( t):=G(\mu(t),u_{\lambda
(a)}(\cdot,t)),$$ where $t\in I_{t_0, \delta}$, $\mu=(a,b)$ and
$G(\mu,u)$ is the same as in the proof of Proposition
~\ref{Prop:Splitting}. The orthogonality conditions on the
fluctuation can be written as $G_\#(\mu,u)=0$. Using the implicit
function theorem we will first prove that for any $\mu_0:=(a_0,
b_0)\in \mathcal{A}_{t_0, \delta}\times \mathcal{B}_{t_0, \delta,
\epsilon_0}$ there exists a neighborhood $\mathcal{U}_{\mu_0}$ of
$V_{\mu_0}$ and a unique $C^1$ map
$g_\#:\mathcal{U}_{\mu_0}\rightarrow \mathcal{A}_{t_0, \delta}\times
\mathcal{B}_{t_0, \delta, \epsilon_0}$ such that $G_\#(g_\#(v),v)=0$
for all $v\in \mathcal{U}_{\mu_0}$.

We claim that $\p_\mu G_\#(\mu,u)$ is invertible, provided
$u_{\lambda (a)}$ is close to $V_\mu$. We compute
\begin{equation}\label{linearization1}
\p_\mu G_\#(\mu,u)( t)=\p_\mu G(\mu(t), u_{\lambda (a)}(\cdot,t))=A(
t) + B( t),
\end{equation}
where
\begin{equation} \label{linearization2}
A( t):= \p_{\mu}G(v,\mu)|_{v= u_{\lambda(a)}},\ B(t):=
\p_vG(v,\mu)|_{v= u_{\lambda(a)}}\p_\mu u_{\lambda (a)}.
\end{equation}
Note that in \eqref{linearization2} $\p_v G(\mu, v)|_{v= u_{\lambda
(a)}}$ is acting on $\p_\mu u_{\lambda (a)}$ as an integral w.r. to
$y$. We have shown in the proof of Proposition ~\ref{Prop:Splitting}
that the first term on the r.h.s. is invertible, provided
$u_{\lambda (a)}$ is close to $V_\mu$.

Now we show that for $\delta>0$ sufficiently small the second term
on the r.h.s. is small. Let $v:= u_{\lambda (a)}$. Assuming for the
moment that $v$ is differentiable, we compute
$\p_av=\p_a(\lambda)\lambda^{-1}[-v+y\p_yv].$ Furthermore,
$\p_a(\lambda)\alpha=-\lambda^{-1}(t)\int_{0}^{t}\alpha(s)ds. \, \,
$ Combining the last two equations together with Equation
(~\ref{linearization2}) we obtain
\begin{equation*}
[B( t)\alpha](t)=-\int {B(t)(y) (-v+y\p_yv)(y, t) dy}\ \lambda
^{-2}(t)\int_{0}^{t}\alpha(s)ds.
\end{equation*}
Integrating by parts the second term in parenthesis gives
\begin{equation}
[B( t) \alpha](t)=\lambda^{-2}(t)\int_{0}^{t}\alpha(s)ds\int {(1
+\p_y\cdot y) B( t)(y)  v(y, t) dy}.
\end{equation}
Now, using a density, or any other, argument we remove the
assumption  of the differentiability on $v$ and conclude that this
expression holds without this assumption. Using this expression and
the inequality $\lambda(t) \geq \sqrt{2}\lambda_0$, provided $\delta
\le (4 \sup a)^{-1}\lambda_0^{2} \le 1/4 \lambda_0^{2}$, we estimate
\begin{equation}
\|B( t) \alpha\|_{L^\infty([t_0,t_0+\delta])}\lesssim \delta
\lambda_0^{-2}\|v
\|_{L^\infty}\|\alpha\|_{L^\infty([t_0,t_0+\delta])}.
\end{equation}
So $B( t)$ is small, if $\delta \lesssim (\lambda_0^{-2}\|v
\|_{L^\infty})^{-1}$, as claimed. This shows that $\p_\mu
G_\#(\mu,u)$ is invertible, provided $u_{\lambda (a)}$ is close to
$V_\mu$. Proceeding as in the proof of Proposition
\ref{Prop:Splitting} we conclude
the proof of Proposition \ref{Prop:LMSplitting2}.
\end{proof}

We say that $\lambda(t)$ is admissible on $I_{t_0, \delta}$ if
$\lambda\in C^{2}(I_{t_0, \delta},\mathbb{R}^{+})\ \mbox{and}\
-\lambda\partial_{t}\lambda \in [1/4, 1]$.
\begin{lemma} \label{induction}
Assume $u \in C^{1}([0,t_*), \langle x\rangle^3 L^\infty)$ and
$\inf_{x\in \mathbb{R}}u(\cdot,t)>0$. Furthermore, assume there is a
$t_0 \in [0, t_*)$ and $u_{\lambda_0} (\cdot, t_0) \in
U_{\epsilon_0/2}$ for some $\lambda_0$ and for $\epsilon_0$ given in
Proposition ~\ref{Prop:Splitting}. Then there are $\delta = \delta
(\lambda_0, u)>0$ and $\lambda (t)$, admissible on $I_{t_0,
\delta}$, s.t. \eqref{eqn:splitting2} and
\eqref{eqn:splitting2conditions} hold on $I_{t_0, \delta}$.
\end{lemma}
\begin{proof}
The conditions $u \in C^{1}([0,t_*), \langle x\rangle^3L^\infty),$
$\inf_{x\in \mathbb{R}}u(\cdot,t)>0$ and $u_{\lambda_0} (t_0) \in
U_{\epsilon_0/2}$ imply that there is a $\delta = \delta
(\lambda_0,u)$ s.t. $u \in U_{t_0, \delta, \epsilon_0, \lambda_0}$.
By Lemma \ref{Prop:LMSplitting2}, the latter inclusion implies that
there is $\lambda (t)$, admissible on $I_{t_0, \delta},\ \lambda
(t_0)= \lambda_0, $ s.t. \eqref{eqn:splitting2} and
\eqref{eqn:splitting2conditions} hold on $I_{t_0, \delta}$.
\end{proof}


\section{A priori Estimates}\label{SEC:ApriEst}
In this section we assume that $u(x,t)$ is a solution to
(~\ref{eq:MCF}) satisfying the following conditions
\begin{itemize}
\item[(A)]
For $0\leq t\leq t_{\#}$ there exist $C^{1}$ functions $a(t)$ and
$b(t)$ such that $u(x,t)$ can be represented as
\begin{equation}\label{eqn:split2}
u(x,t)=\lambda(t)[\lb\frac{2(d-1)+b(t)y^{2}}{a(t)+\frac{1}{2}}\rb^{\frac{1}{2}}+\phi(y,\tau)]
\end{equation}
where $\phi(\cdot,\tau)\perp e^{-\frac{a(t)}{2} y^2},\
(1-a(t)y^{2})e^{-\frac{a(t)}{2} y^2}$ (see (~\ref{eqn:split})),
$y=\lambda^{-1}(t)x$ and $\tau(t):= \int_{0}^{t}\lambda^{-2}(s)ds$,
$-\lambda(t)\partial_{t}\lambda(t)=a(t)$.
\end{itemize}

In the following we define estimating functions to control the
functions $\phi(y,\tau),$ $a(t(\tau))$ and $b(t(\tau)).$
\begin{equation}\label{eq:majorants}
\begin{array}{lll}
M_{m,n}(T):=\displaystyle\max_{\tau\leq
T}\beta^{-\frac{m+n+1}{2}}(\tau)\|
\phi(\cdot,\tau)\|_{m,n},\\
A(T):=\displaystyle\max_{\tau\leq T}\beta^{-2}(\tau)|a(t(\tau))-\frac{1}{2}+\frac{1}{d-1}b(t(\tau))|,\\
B(T):=\displaystyle\max_{\tau\leq
T}\beta^{-7/4}(\tau)|b(t(\tau))-\beta(\tau)|.
\end{array}
\end{equation} with $(m,n)=(3,0),\ (\frac{11}{10},0),\ (2,1), \ (1,2)$ and with the function $\beta(\tau)$ defined as
\begin{equation}\label{FunBTau}
\beta(\tau):=\frac{1}{\frac{1}{b(0)}+\frac{\tau}{d-1}}.
\end{equation}
Furthermore
we define a vector $M$ as
\begin{equation}\label{eq:vector}
M:=(M_{i,j}), \ (i,j)=(3,0),\ (\frac{11}{10},0),\ (1,2),\ (2,1)
\end{equation} and its sum $|M|:=\sum_{i,j}M_{i,j}$.

We say that a polynomial $P(M,A)$ is monotonically nondecreasing if
$P(M_{1},A_{1})\geq P(M_{2},A_{2})$ whenever $A_{1}\geq A_{2}$ and
$M^{(1)}_{i,j}\geq M^{(2)}_{i,j}$ for all $i,j$, with
$M_{1}:=(M^{(1)}_{i,j}), \ M_{2}:=(M^{(2)}_{i,j}).$ In what follows
the symbols $P(M,A)$ and $P(M)$ to stand for different monotonically
nondecreasing polynomials of the vector $M$ and the variable $A$.

In this section we present a priori bounds on the fluctuation $\phi$
proved in later sections.

\begin{proposition}\label{Pro:Main2}
Suppose that $u(x,t)$ is a solution to (~\ref{eq:MCF}) satisfying
Condition (A) and its datum $u_{0}(x)$ satisfies all the conditions
in Theorem ~\ref{maintheorem} except the ones in Statement (4). Let
the parameters $a(t)$, $b(t)$ and the function $\phi(y,\tau)$ be the
same as in (~\ref{eqn:split2}). Then there exists a nondecreasing
polynomial $P(M,Z)$ of the 4-vector $M$ and variable $A$ such that
the functions $a$, $b$ and $\phi$ satisfy the estimates
\begin{equation}\label{eq:B}
B(\tau) \lesssim
1+P(M(\tau),A(\tau)),
\end{equation}
\begin{equation}\label{eq:A}
A(\tau)\lesssim A(0)+1+\beta(0)P(M(\tau),A(\tau)),
\end{equation}
\begin{equation}\label{eq:M30}
M_{3,0}(\tau)\lesssim M_{3,0}(0)+\beta^{ \frac{1}{2}
}(0)P(M(\tau),A(\tau)),
\end{equation}
\begin{equation}\label{eq:M20}
M_{\frac{11}{10},0}(\tau)\lesssim
M_{\frac{11}{10},0}(0)+M_{3,0}(\tau)+ \beta^{ \frac{1}{2}
}(0)P(M(\tau),A(\tau)),
\end{equation}
\begin{equation}\label{eq:M21}
M_{2,1}(\tau) \lesssim M_{2,1}(0)+M_{3,0}(\tau)+\beta^{ \frac{1}{2}
}(0)P(M(\tau),A(\tau)),
\end{equation}
\begin{equation}\label{eq:M12}
M_{1,2}(\tau)\lesssim M_{1,2}(0)+M_{3,0}(\tau)+M_{2,1}(\tau)+\beta^{
\frac{1}{2} }(0)P(M(\tau),A(\tau)),
\end{equation}
for any $\tau\in[0,\tau(t_{\#})]$ provided that $v(y,\tau)\geq
\frac{1}{4}\sqrt{2(d-1)}$ $$v(\cdot,\tau)\in \langle y\rangle
L^{\infty},\
\partial_{y}v(\cdot,\tau),\ \partial_{y}^{2}v(\cdot,\tau)\in
L^{\infty},|M(\tau)|\ \text{and}\ A(\tau),\ B(\tau)\leq \beta^{-
\frac{1}{4} }(\tau) $$ in the same interval $\tau\in
[0,\tau(t_{\#})]$.
\end{proposition}
The proof of Equations (~\ref{eq:B}) and (~\ref{eq:A}) is given in
Section ~\ref{SEC:EstB}. Equation (~\ref{eq:M30}), (~\ref{eq:M20}),
(~\ref{eq:M21})  and (~\ref{eq:M12}) are proved in Sections
~\ref{SEC:EstM1}, ~\ref{SEC:EstM2}, ~\ref{SEC:EstM21} and
~\ref{SEC:estM12} respectively.

\begin{corollary}\label{cor:aprior}
Let $\phi$ be defined in (~\ref{eqn:split2}) and assume $|M(0)|,\
A(0),\ B(0)\lesssim 1$. Assume there exists an interval $[0,T]$ such
that for $\tau\in [0,T]$,
$$
|M(\tau)|, A(\tau),\ B(\tau)\leq \beta^{-\frac{1}{4}}(\tau).$$
Then on the same time interval the parameters $a$, $b$ and the
function $\phi$ satisfy the following estimates
\begin{equation}\label{EstABM}
|M(\tau)|,\ A(\tau),\ B(\tau)\lesssim 1.
\end{equation}
\end{corollary}
\begin{proof}
By replacing $M_{3,0}(\tau),\ M_{2,1}(\tau)$ on the right hand sides
of (~\ref{eq:M20})-(~\ref{eq:M12}) by the estimates (~\ref{eq:M30})
(~\ref{eq:M21}) we rewrite Equations (~\ref{eq:M30})-(~\ref{eq:M12})
as $$ A(\tau)+|M(\tau)|\lesssim
A(0)+1+|M(0)|+\beta^{1/2}(0)P(|M(\tau)|,A(\tau))
$$ with $P$ being some polynomial. This implies (~\ref{EstABM}) by
the assumptions on $|M(0)|,\ A(0)$ and $B(0).$
\end{proof}
\section{Lower and Upper Bounds of $v$}\label{SEC:LowerBound}
In this section we prove lower and upper bounds for $v$ defined in
(~\ref{eq:definev}). The main tool we use is a generalized form of
maximum principle from ~\cite{LSU}.
\begin{lemma}\label{LM:max}
Suppose $u(y,\tau)$ is a smooth function satisfying the
estimates
$$u_{\tau}-a_{0}(y,\tau)u_{yy}-[a_{1}(y,\tau)+m(\tau)y]u_{y}-a_{2}(y,\tau)u\leq 0;$$
$$\langle y\rangle^{-1} u(y,\tau)\in L^{\infty};$$
\begin{equation}\label{eq:condiMax}
u(y,0)\leq 0 \ \text{if}\ |y|\geq c(0)\ \text{and}\ u(y, \tau)\leq 0
\ \text{if}\ \tau\leq T\ \text{and}\ |y|=c(\tau)
\end{equation}
for some smooth, bounded functions $a_{0}, a_{1}, a_{2}, m, c,$ such
that $a_{0}(y,\tau)\geq 0$ and $c(\tau)\geq 0.$ Then for any
$\tau\leq T$
\begin{equation}\label{eq:resMax}
 u(y,\tau)\leq 0\ \text{if}\ |y|\geq c(\tau).
\end{equation} Moreover, if we replace the condition (~\ref{eq:condiMax})
by the condition that $u(y,0)\leq 0$ for any $y$, then instead of
(~\ref{eq:resMax}) we have $u(y,\tau)\leq 0.$
\end{lemma}
\begin{proof}
In what follows we only prove the estimate for the region $|y|\geq
c(\tau),$ the estimate for $y\in \mathbb{R}$ is almost the same. We start
with transforming the function $u$ so that the standard maximum
principle can be used. Define a new function $w$ by
\begin{equation}\label{Eq:defineV2}
e^{\kappa \tau}\langle z\rangle w(z,\tau):=u(y,\tau)
\end{equation} with $z:=y
e^{\int_{0}^{\tau}m(s)ds}$ and the scalar $\kappa$ to be chosen
later. Then $w$ is a smooth, bounded function satisfying the
inequality
$$w_{\tau}-a_{3}(z,\tau)w_{zz}-a_{4}(z,\tau)w_{z}-a_{5}(z,\tau)w\leq 0$$
for some bounded, smooth functions $a_{3}, \ a_{4}, a_{5}$ and
especially $a_{3},\ a_{5}\geq 0$ by choosing appropriate $\kappa.$
Moreover
$$w(z,0)\leq 0 \ \text{for}\ |z|\geq c(0),\ \text{and}\
w(z, \tau)\leq 0 \ \text{for}\ \tau\leq T,\ |z|=c(\tau)
e^{\int_{0}^{\tau}m(s)ds}.$$ By the standard maximum principle we
have
$$w(z,\tau)\leq 0\ \text{if}\ \tau\leq T\ \text{and}\ |z|\geq
c(\tau)e^{\int_{0}^{\tau}m(s)ds}.$$ This estimate and the relation
between $w$ and $u$ in (~\ref{Eq:defineV2}) imply the desired result
in the case $|y|\geq c(\tau)$.
\end{proof} Recall the definition of function $g(y,\beta)$ from (\ref{eq:Lower}). The following proposition plays an important
role in our analysis.
\begin{proposition}\label{PROP:UppLow} Assume $v$ satisfies Condition (A) in Section ~\ref{SEC:ApriEst},
$v(y,0)\geq g(y,b_{0})$, $v(y,0)\in \langle y\rangle L^{\infty},$
$|\partial_{y}v(y,0)v^{-1/2}(y,0)|\leq \kappa_{0}\beta^{1/2}(0)$ and
$|\partial_{y}^{n}v(y,0)|\leq \kappa_{0}\beta^{\frac{n}{2}}(0),\
n=2,3,4$, and assume there exists a time $\tau(t_{\#})\geq
\tau_{1}>0$ such that for any $\tau\leq \tau_{1}$, $|M(\tau)|$,
$A(\tau), B(\tau)\leq \beta^{-\frac{1}{4}}(\tau)$ and
\begin{equation}\label{eq:space} v(\cdot,\tau)\in \langle y\rangle
L^{\infty},\
\partial_{y}v(\cdot,\tau),\ \partial_{y}^{2}v(\cdot,\tau)\in
L^{\infty},\ \text{and}\ v(y,\tau) \geq c(\tau)
\end{equation} for some
$c(\tau)> 0.$ Then we have
\begin{equation}\label{eq:comparison}
v(y,\tau)\geq g(y,\beta(\tau)),\ |{v^{- \frac{1}{2}
}(y,\tau)}\partial_{y}v(y,\tau)|\lesssim \beta^{ \frac{1}{2}
}(\tau),\ |\partial_{y}^{n}v(y,\tau)|\lesssim
\beta^{\frac{n}{2}}(\tau),\ n=2,3,4
\end{equation} on the
same interval.
\end{proposition}
\begin{proof}
We start with proving the first estimate in (~\ref{eq:comparison})
by verifying that the equation for $v$ satisfies all the conditions
in Lemma ~\ref{LM:max}. Since $M_{3,0}(\tau),\ A(\tau),\ B(\tau)\leq
\beta^{-\frac{1}{4}}(\tau)$ and
$v(y,\tau)=\lb\frac{2(d-1)+b(\tau)y^{2}}{a(\tau)+\frac{1}{2}}\rb^{\frac{1}{2}}+\phi(y,\tau)$
we have
\begin{equation}
v(y_{1},\tau)\geq g(y_{1},\tau)\label{eq:comparison0}
\end{equation} for $y_{1}$ satisfying $|\beta y^{2}_{1}|\leq 20(d-1)$ and $\tau\in
[0,\tau_{1}].$

On the other hand we observe that
$a=\frac{1}{2}-\beta+O(\beta^{\frac{3}{2}})$ by the assumption on
$A(\tau)$. By a direct computation we have that on the domain $
[0,\tau_{1}]\times\{|\beta y^{2}|\geq 20(d-1)\}$
\begin{equation}\label{eq:function}
H(g)\leq 0\ \text{and}\ H(v)=0
\end{equation} where the map $H(g)$
is defined as
$$H(g)=g_{\tau}-\frac{g_{yy}}{1+g_{y}^{2}}+\frac{d-1}{g}+ay\partial_{y}g-ag.$$

In order to use Lemma ~\ref{LM:max} derive an equation for $g-v$. By
the forms of $H(g)$ and $H(v)$ there exist functions $b_{n}, \
n=1,2,3,$ such that
\begin{equation}\label{eq:BeforeMaxPrinciple}
\partial_{\tau}(g-v)-b_{1}\partial_{y}^{2}(g-v)+ay\partial_{y}(g-v)-b_{2}\partial_{y}(g-v)-b_{3}(g-v)=H(g)-H(v)
\end{equation} where $b_{1}>0$ and $b_{n}, \ n=1,2,3,$ are bounded functions.

Equations (~\ref{eq:comparison0})-(~\ref{eq:BeforeMaxPrinciple}),
the condition (~\ref{eq:space}) and the assumption $v(y,0)\geq
g(y,b_{0})$ enable us to use Lemma ~\ref{LM:max} on the equation for
$g-v$. This leads to the inequality
\begin{equation}\label{eq:lower2}
v(y,\tau)\geq g(y,\beta(\tau))\ \text{if}\ \beta y^{2}\geq 20(d-1).
\end{equation}

For the region $\beta y^{2}\leq 20(d-1)$ we use the estimate
(~\ref{eq:comparison0}), which together with (~\ref{eq:lower2})
yields the first estimate in (~\ref{eq:comparison}).

Now we prove the estimate on $\partial_{y}v.$ Differentiating
Equation (~\ref{eqn:BVNLH}) we obtain an equation for
$\partial_{y}v$. However this equation is not accessible directly to
a maximum principle. To overcome this problem we use the fact that
$v$ is large for large $|y|$ which is proved above and use instead
the equation for $v^{- \frac{1}{2} }\partial_{y}v$. We define a new
function $h(y,\tau):=v^{ \frac{1}{2} }(y,\tau).$ Then
$\partial_{y}h=\frac{1}{2}v^{-\frac{1}{2} }{\partial_{y}v}$
satisfies the equation
$$\mathcal{K}(\partial_{y}h)=0,$$ where the map $\mathcal{K}(\chi)$
is defined as $$
\begin{array}{lll}
\mathcal{K}(\chi)&:=&\frac{d}{d\tau}\chi-\frac{1}{1+(\partial_{y}v)^{2}}\partial_{y}^{2}\chi
+[\frac{v^{-1}
(\partial_{y}v)^{2}+2\partial_{y}^{2}v}{1+(\partial_{y}v)^{2}}-\frac{1}{v}]\frac{\partial_{y}v}{1+(\partial_{y}v)^{2}}
\partial_{y}\chi+ay\partial_{y}\chi+\frac{a}{2}\chi\\
&
&-\frac{3(d-1)}{2v^{2}}\chi+\frac{1}{v(1+(\partial_{y}v)^{2})}\chi^{3}
+\frac{8}{(1+(\partial_{y}v)^{2})^{2}}\chi^{5}.
\end{array}
$$ On the other hand since $\frac{a}{2}>\frac{1}{5}$, implied by the assumption on $A$,
and $h^{2}=v\geq 4\sqrt{d-1}$ on the region $\beta y^{2}\geq
20(d-1)$, we have that $$\mathcal{K}(\kappa_{0}\beta^{ \frac{1}{2}
})>0\ \text{and}\ \mathcal{K}(-\kappa_{0}\beta^{ \frac{1}{2} })<0$$
provided that $b(0)>0$, (and therefore $\beta(\tau)\leq b(0)$), is
sufficiently small. Recall that the constant $\kappa_{0}>2$ defined
in Theorem ~\ref{maintheorem}. Moreover, by the assumption
$M_{2,1}\leq \beta^{- \frac{1}{4} }$ we have
$$-\kappa_{0}\beta^{\frac{1}{2}}(\tau)<\partial_{y}h(y,\tau)|_{\beta(\tau)y^{2}=
20(d-1)}<\kappa_{0}\beta^{\frac{1}{2}}(\tau).$$ By the condition on
$v(y,0)$ we have that
$$-\kappa_{0}\beta^{\frac{1}{2}}(0)<\partial_{y}h(y,0)|_{\beta(0)y^{2}\geq
20(d-1)}<\kappa_{0}\beta^{\frac{1}{2}}(0).$$

Lastly we derive equations for $\partial_{y}h\pm
\kappa_{0}\beta^{1/2}$ from
$\mathcal{K}(\partial_{y}h)-\mathcal{K}(\mp \kappa_{0}\beta^{1/2})$
whose proof is almost identical to (~\ref{eq:BeforeMaxPrinciple}),
thus omitted.

Collecting the facts above and using that $h=\sqrt{v}$, we have by
Lemma ~\ref{LM:max} the second part of (~\ref{eq:comparison}).

By almost the same reasoning on the equation for $\partial_{y}^{2}v$
we prove that $|\partial_{y}^{2}v|\lesssim \beta$.

Next, we present the proof of the estimate on $\partial_{y}^{3}v$.
The estimates for $\partial_{y}^{4}v$ is easier, thus omitted. We
compute to get $$W(\partial_{y}^{3}v)=g_{1}$$ where the map $W(h)$
is defined as
$$W(h):={\partial_\tau}h-\frac{1}{1+v_{y}^{2}}\partial_{y}^{2}h-g_{4}\partial_{y}h+ay\partial_{y}h+2ah-\frac{d-1}{v^{2}}h+\frac{6\partial_{y}v}{[1+(v_{y})^{2}]^{2}}h^{2}-g_{3}h,$$
where $g_{4}$ is a function of $\partial_{y}^{n}v,\ n=1,2,3$, and
$g_{1}, g_{3}$ are functions of $v^{- \frac{1}{2} }\partial_{y}v$
and $\partial_{y}^{2}v.$ Moreover $$g_{4}\in L^{\infty},\
|g_{3}|\lesssim \beta^{ \frac{1}{2} },\
g_{1}=g_{1}(\frac{\partial_{y}v}{\sqrt{v}},\partial_{y}^{2}v)\
\text{satisfies} \ \|g_{1}\|_{\infty}\leq \kappa\beta^{2/3}$$ for
some constant $\kappa>0$ by the facts $\partial_{y}^{n}v\in
L^{\infty}, n=1,2,3,$ and their various estimates above. Recall that
$\|\partial_{y}^{3}v(\cdot,0)\|_{\infty}\leq
\kappa_{0}\beta^{3/2}(0)$ for some $\kappa_{0}>0$. We define a new
constant $\kappa_{1}:=\max\{\kappa_{0},\kappa\}.$ By the assumption
on $A$ we have $-2a-\frac{d-1}{v^{2}}+g_{3}\leq - \frac{1}{2} $
hence
$$W(2\kappa_{1}\beta^{3/2})>0\ \text{and}\
W(-2\kappa_{1}\beta^{3/2})<0.$$ As in (~\ref{eq:BeforeMaxPrinciple})
we derive equations for $\partial_{y}^{3}v\mp
2\kappa_{1}\beta^{3/2}$ from $W(\partial_{y}^{3}v)-W(\pm
2\kappa_{1}\beta^{3/2})$, on which we use the maximum principle to
have the estimate for $\partial_{y}^{3}v.$

The proof is complete.
\end{proof}
The following proposition is used in the proof of the statement (4)
of Theorem ~\ref{maintheorem}. We define a function $\varrho$ as
\begin{equation}\label{eq:defineRho}
\varrho(y,\tau):=\frac{v\partial_{y}^{2}v}{1+(\partial_{y}v)^{2}}
=\frac{u(x,t)\partial_{x}^{2}u(x,t)}{1+(\partial_{x}u)^{2}}
\end{equation} where
the last equality follows from the definition of $v.$
\begin{proposition}\label{Prop:Comparison}
Suppose that $v$ satisfies all the conditions in Proposition
~\ref{PROP:UppLow}. The we have
\begin{equation}\label{eq:Rate1}
|\varrho(y,\tau)|\leq 4\beta(\tau)\ \text{for}\ y\in
[-\frac{1}{10},\frac{1}{10}]
\end{equation} and if $v_{0}\partial_{y}^{2}v_{0}\geq -1$ then
\begin{equation}\label{eq:Rate2}
\varrho(y,\tau)\geq -1\ \text{for}\ \beta y^{2}\geq 2(d-1).
\end{equation}
If $\rho(\cdot,0)\leq d-1$ then
\begin{equation}\label{eq:meancurvature}
\rho(\cdot,\tau)\leq d-1.
\end{equation}
\end{proposition}
\begin{proof}
Recall that $v(y,\tau)=V_{a,b}+\phi(y,\tau)$ with
$|\phi(y,\tau)|\leq \beta^{\frac{7}{4}}\langle y\rangle^{3}$ by the
assumption on $M_{3,0}\leq \beta^{-\frac{1}{4}}$. This implies
(~\ref{eq:Rate1}) and that $\varrho(y,\tau)\geq -1 $ when $\beta
y^{2}=2(d-1).$

We use the maximum principle to prove (~\ref{eq:Rate2}). First we
derive an inequality for the function
$\varpi:=\frac{\partial_{x}^{2}uu}{1+(\partial_{x}u)^{2}}+1.$ We
show below that
\begin{equation}\label{comparison}
\mathcal{Y}_{u}(\varpi)=0
\end{equation} where the
linear mapping $\mathcal{Y}_{u}$ is defined as
$$\mathcal{Y}_{u}(\psi)=\partial_{t}\psi-\frac{1}{1+(\partial_{x}u)^{2}}\partial_{x}^{2}\psi
+\frac{2\partial_{x}u}{u[1+(\partial_{x}u)^{2}]}[\frac{\partial_{x}^{2}u u}{1+(\partial_{x}u)^{2}}+1]
\partial_{x}\psi-\frac{2(\partial_{x}u)^{2}}{u^{2}(1+(\partial_{x}u)^{2})}\chi\psi$$
where the function $\chi$ is defined as
$\chi(x,t):=\frac{\partial_{x}^{2}uu}{1+(\partial_{x}u)^{2}}-(d-1).$
Recall that $\varpi(x,t)=\varrho(y,\tau)+1\geq 0$ if $\beta(\tau(t))
y^{2}= 2(d-1)$ or $t=0.$ Using the maximum principle Lemma
~\ref{LM:max} on (~\ref{comparison}) we have that $\varpi\geq 0$ for
$\beta y^{2}\geq 2(d-1)$, which is (~\ref{eq:Rate2}).

Now Equation (~\ref{comparison}) follows from considering
$\mathcal{K}_{u}(\chi(\cdot,t))-\mathcal{K}_{u}(-d)$, as in
(~\ref{eq:BeforeMaxPrinciple}), and the observations that
$$\mathcal{K}_{u}(\chi(\cdot,t))=0\ \text{and}\ \mathcal{K}_{u}(-d)= 0$$ where the map $\mathcal{K}_{u}$
is defined as
$$\mathcal{K}_{u}(h):=\partial_{t}h-\frac{1}{1+(\partial_{x}u)^{2}}\partial_{x}^{2}h
+\frac{2\partial_{x}u}{u(1+(\partial_{x}u)^{2})}(h+d)\partial_{x}h
-\frac{2(\partial_{x}u)^{2}}{u^{2}(1+(\partial_{x}u)^{2})}(h+d)h.$$

The proof of (~\ref{eq:meancurvature}) is similar by using the
observations $$\mathcal{K}_{u}(\chi(\cdot,t))=0\ \text{and}\
\mathcal{K}_{u}(0)= 0$$ and the initial condition
$\chi(\cdot,0)=\rho(\cdot,0)\leq 0.$

This completes the proof of Proposition ~\ref{Prop:Comparison}.
\end{proof}
\section{Proof of Main Theorem ~\ref{maintheorem}}\label{SecMain}

Choose $b_0$ so that $Cb_{0}^2\leq \frac{1}{2}\epsilon_{0}$ with $C$
the same as in \eqref{eq:INI2} and with $\epsilon_{0}$ given in
Proposition ~\ref{Prop:Splitting}. Let
$v_{0}(y):=\lambda_{0}^{-1}u_{0}(\lambda_{0}y).$ Then $v_{0}\in
U_{\frac{1}{2}\epsilon_{0}}$, by the condition \eqref{eq:INI2} on
the initial conditions with $(m,n)=(3,0)$. Hence Proposition
\ref{Prop:Splitting}  holds for $v_{0}$ and we have the splitting
$v_{0}=V_{g(v_{0})}+\eta_{0}$. Denote $g(v_0) =: (a(0), b(0))$.


Furthermore, by Lemma \ref{induction} there are $\delta_1 > 0$ and
$\lambda_1 (t)$, admissible on $[0, \delta_1]$, s.t. $\lambda_1 (0)=
\lambda_0$ and Equations \eqref{eqn:splitting2} and
\eqref{eqn:splitting2conditions} hold on the interval $[0,
\delta_{1}]$. Hence, in particular, the estimating functions
$M(\tau)=(M_{m,n}(\tau))$, $(m,n)=(3,0),\ (\frac{11}{10},0),\
(2,1),\ (1,2)$,\ $A(\tau)$ and $B(\tau)$ of Section 5 are defined on
the interval $[0, \delta_1]$. We will write these functions in the
original time $t$, i.e. we will write $M(t)$ for $M(\tau (t))$ where
$\tau(t)=\int_{0}^{t}\lambda_{1}^{-2}(s)ds$.

Recall the definitions of $\beta(\tau)$ and $\kappa$ are given in
(~\ref{FunBTau}). By the relation $\beta(0)=b(0)$, Equation
(~\ref{eq:INI2}) and Proposition ~\ref{Prop:SplittingIC},   $A(0)$,
$|M(0)|\lesssim 1$,
while $B(0)=0$ and $v_{0}(y)\geq g(y,b_{0})$, by the definition. We
have, by the continuity,
that for a sufficiently small time interval, which we can take to be
$[0, \delta_1]$, Condition (A) in Section ~\ref{SEC:ApriEst} holds
and
\begin{equation}\label{ApriorEST}
|M(t)|,\ A(t), \ B(t)\leq \beta^{-\frac{1}{4}}(\tau(t)),\
\text{and}\ u_{\lambda_{1}}(\cdot,t)\geq
 \frac{1}{4}\sqrt{2(d-1)},
\end{equation} the last fact together with Theorem
~\ref{THM:WellPose} and the initial conditions implies that
$$v(\cdot,\tau)\in \langle y\rangle L^{\infty},\
\partial_{y}v(\cdot,\tau),\ \partial_{y}^{2}v(\cdot,\tau)\in
L^{\infty}.$$ Then by Proposition ~\ref{Pro:Main2}, Corollary
\ref{cor:aprior} we have that for the same time interval
\begin{equation}\label{UtimateEST}
|M(t)|,\ A(t),\ B(t)\lesssim 1,\ \text{and}\
u_{\lambda_{1}}(\cdot,t)\geq g(y,\beta(\tau)).
\end{equation}

Equation \eqref{UtimateEST} implies that $u_{\lambda_1} (\cdot,
\delta_1) \in U_{\epsilon_0/2}$ and
$u_{\lambda_{1}}(\cdot,\delta_{1})\geq g(y,\beta(\tau(t)))$ (indeed,
by the definition of $M_{3,0}(t)$ we have $\|u_{\lambda_{1}}
(\cdot,t)-V_{a(t),b(t)}\|_{3,0}\leq M_{3,0}(t) b^{2}(t)$). Now we
can apply Lemma \ref{induction} again and find $\delta_2 > 0$ and
$\lambda_2 (t)$, admissible on $[0, \delta_1 +\delta_2]$, s.t.
$\lambda_2 (t)= \lambda_1 (t)$ for $t \in [0, \delta_1]$ and
Equations \eqref{eqn:splitting2} and
\eqref{eqn:splitting2conditions} hold on the interval $[0, \delta_1
+\delta_2]$.

We iterate the procedure above to show that there is a maximal time
$t^* \le t_*$  (the maximal existence time), and a function
$\lambda(t)$, admissible on $[0, t^*)$, s.t. \eqref{eqn:splitting2}
and \eqref{eqn:splitting2conditions} and \eqref{UtimateEST} hold on
$[0, t^*)$. We claim that $t^* = t_*$ and $t^* < \infty$ and,
consequently, $\lambda(t^*) = 0$. Indeed, if $t^* < t_*$ and
$\lambda(t^*) > 0$, then by the a priori estimate $u_{\lambda} (t)
\in U_{\epsilon_0/2}$ and $u(x,t)\geq
\frac{1}{4}\lambda(t)\sqrt{2(d-1)}$ for any $t \le t^*.$ By Lemma
\ref{induction}, this implies that there is $\delta
>0$ and $\lambda_{\#}(t)$, admissible on $[0, t^*+\delta]$, s.t.
\eqref{eqn:splitting2} and \eqref{eqn:splitting2conditions} hold on
$[0, t^*+\delta]$ and $\lambda_{\#}(t) = \lambda(t)$ on $[0, t^*)$,
which would contradict the assumption that the time $t^*$ is
maximal. Hence
\begin{equation}\label{eq:TwoCases}
\text{either}\ t^* = t_*\ \text{or}\ t^* < t_*\ \text{and}\
\lambda(t^*) = 0.
\end{equation}
The second case is ruled out as follows. Using the relation between
the functions $u(x,t)$ and $v(y,\tau)$ and Equation
(~\ref{UtimateEST}) we obtain the following a priory estimate on the
(non-rescaled) solution $u(x,t)$ of equation (~\ref{eq:MCF}):
\begin{equation}\label{upperbound}
u(x,t) \geq \lambda(t)g(y,\beta(\tau(t)))\geq
\lambda(t)\frac{1}{4}\sqrt{2(d-1)}.
\end{equation}
Moreover by \eqref{eqn:split2} and the fact $\|\langle
y\rangle^{-3}\phi(y,\tau(t))\|_{\infty}\lesssim b^{2}(t)$ implied by
$M_{1}\lesssim 1$,
\begin{equation}
u(0,t)\le \lambda(t)\lsb \lb\frac{2(d-1)}{ c(t)}\rb^\frac{1}{2}+C
b(t)^2 \rsb\rightarrow 0,
 \label{eqn:sqeefdsfew}
\end{equation}
as $t\uparrow t^*$, which implies that $t^* \ge t_*$ and therefore
$t_* = t^*$ is the collapsing time as claimed.

Now we consider the first case in (~\ref{eq:TwoCases}). In this case
we must have either $t^* = t_* = \infty$ or $t^* = t_* < \infty$ and
$\lambda(t^*) = 0$, since otherwise we would have existence of the
solution on an interval greater than $[0, t_*)$. Finally, the case
$t^* = t_* = \infty$ is ruled out in the next paragraph.
This proves the claim which can reformulated as: there is a function
$\lambda(t)$, admissible on $[0, t_*)$, $t_*<\infty$ s.t.
\eqref{eqn:splitting2} and \eqref{eqn:splitting2conditions} and
\eqref{UtimateEST} hold on $[0, t_*)$ and $\lambda(t) \rightarrow 0$
as $t \rightarrow t_{*}$.


By the definitions of $A(t)$ and $B(t)$ in (~\ref{eq:majorants}) and
the facts that $A(t), B(t)\lesssim 1 $ proved above, we have that
\begin{equation}\label{EstABtau}
a(t)-\frac{1}{2}=- \frac{1}{d-1}b(t)+O(\beta^{2}(\tau(t))),\
b(t)=\beta(\tau(t))(1+O( \beta^{ \frac{1}{2} }(\tau(t)))),
\end{equation}
where, recall, $\tau = \tau(t)=\int_{0}^{t}\lambda^{-2}(s)ds$. Hence
$a(t)-\frac{1}{2}=O(\beta(\tau))$.  Recall that
$a=-\lambda\partial_{t}\lambda$, which can be rewritten as
$\lambda^{2}(t)=\lambda_{0}^{2}-2 \int_{0}^{t}a(s)ds$ or
$\lambda(t)=[\lambda_{0}^{2}-2 \int_{0}^{t}a(s)ds]^{1/2}.$
Since $|a(t)-\frac{1}{2}|=O(b(t))$, there exists a time $t^{*} \le
t^{**}< \infty$ such that $\lambda_{0}^{2}=2
\int_{0}^{t^{**}}a(s)ds$, i.e. $\lambda(t)\rightarrow 0$ as
$t\rightarrow t^{**}$. Furthermore, by the definition of $\tau$ and
the estimate $|a(t)-\frac{1}{2}|=O(b(t))$ we have that
$\tau(t)\rightarrow \infty$ as $t\rightarrow t^{**}$ (precise
expressions are given in the next paragraph). Since $\lambda(t^*) =
0$ we must have $t^{*} = t^{**}$. Thus we have shown that
$t^*<\infty$.

This completes the proof of Statement (1) and (2) of Theorem
~\ref{maintheorem}.

Now we prove Statement (3) of Theorem ~\ref{maintheorem} which
establishes the asymptotics of the parameter functions. Equation
(~\ref{EstABtau}) implies $b(t)\rightarrow 0$ and $a(t)\rightarrow
\frac{1}{2}$ as $t\rightarrow t^*$.  By the analysis above and the
definitions of $a$, $\tau$ and $\beta$ (see \eqref{FunBTau}) we have
$$\lambda(t)=(t^*-t)^{1/2}(1+o(1)),\ \tau(t)=-ln|t^*-t|(1+o(1)),$$
and $$\beta(\tau(t))=-\frac{1}{(d-1)ln|t^*-t|}(1+o(1)).$$ This gives
the first equation in \eqref{eq:para}. By (~\ref{EstABtau}) and the
relation $c=a + \frac{1}{2}$ we have the last two equations in
\eqref{eq:para}. This proves Statement (3) of Theorem
~\ref{maintheorem}.

Now we prove the fourth statement of Theorem ~\ref{maintheorem}. First we show for $x\not=0$ $\displaystyle\underline\lim_{t\rightarrow t_{*}}u(x,t)\geq 0$. We
transform (~\ref{eq:MCF}) as
\begin{equation}\label{eq:transform}
\partial_{t}u=[\frac{u
\partial_{x}^{2}u}{1+(\partial_{x}u)^{2}}-(d-1)]\frac{1}{u}.
\end{equation} By
the estimate (~\ref{eq:Rate1}) and the definition
$\varrho(y,\tau):=\frac{u
\partial_{x}^{2}u}{1+(\partial_{x}u)^{2}}$ we have that
$$ \partial_{t}u(0,t)\leq [4\beta(0)-(d-1)]\frac{1}{u(0,t)},
$$
or
$$u^{2}(0,t)\leq u^{2}(0,t_{1})-2[(d-1)-4\beta(0)](t-t_{1}).$$
This together with the fact $u(0,t_{*})=0$ yields
\begin{equation}\label{eq:0first}
t_{*}-t_{1}\leq \frac{u^{2}(0,t_{1})}{2[(d-1)-4\beta(0)]}.
\end{equation}
On the other hand if a fixed $x_{1}\not=0$ satisfies
$\lambda^{-2}(t)\beta(\tau(t))x_{1}^{2}=\beta(\tau)y_{1}^{2}\geq
20(d-1)$ for $t=t_{1}$, and therefore for all $t\geq t_{1}$ then by
(~\ref{eq:defineRho})(~\ref{eq:Rate2}) and (~\ref{eq:transform}) we have
$$
\partial_{t}u(x_{1},t)\geq -\frac{d}{u(x_{1},t)}, $$ or
\begin{equation}
u^{2}(x_{1},t)\geq u^{2}(x_{1},t_{1})-2d(t-t_{1}).
\end{equation} Now we compare $u(x_{1},t_{1})$ and
$u(0,t_{1})$ to see that $u(0,t)$ goes to zero first. Recall that
$u(x,t)=\lambda(t) v(y,\tau)$ and $v(y,\tau)$ has the lower bound
$g(y,b(\tau))$ defined in (~\ref{eq:Lower}), moreover the estimate
$M_{3,0}\lesssim 1$ implies $u(0,t)=\lambda(t) v(0,\tau(t))\leq
\lambda(t) 2\sqrt{d-1}$, thus we have
\begin{equation}\label{eq:compare3}
u(x_{1},t_{1})\geq 2 u(0,t_{1}).
\end{equation} Equations (~\ref{eq:0first})-(~\ref{eq:compare3}) and the fact
$d\geq 2$ yield for any $t<t_{*}$
$$u^{2}(x_{1},t)\geq \frac{3d-4-16\beta(0)}{d-1-4\beta(0)}u^{2}(0,t_{1})\geq u^{2}(0,t_{1})>0$$ i.e. there exists a constant $u_*(x_{1})>0$ such that
$u(x_{1},t)\geq u_*(x_{1})$ before the collapsing time, i.e. $t<
t^*.$

Moreover, by the fast decay of $\lambda(t)$ and slow decay of
$\beta(\tau(t))$ we have that for any $x_{1}\not=0,$ there exists a
time $t_{1}$ such that
$\lambda^{-2}(t)\beta(\tau(t))x_{1}^{2}=\beta(\tau)y_{1}^{2}\geq
20(d-1)$ for $t\geq t_{1}$. This implies that there exists a
$u_*(x_{1})$ such that \begin{equation}\label{eq:nonzero}
u(x_{1},t)\geq u_*(x_{1})>0
\end{equation}
for any $x_{1}\not=0$ which is the first part of Statement (4).

If
$\frac{u_{0}\partial_{x}^{2}u_{0}}{1+(\partial_{x}u_{0})^{2}}-(d-1)\leq
0$ then by (~\ref{eq:meancurvature}) we have
$\partial_{t}u=[\frac{u\partial_{x}^{2}u}{1+(\partial_{x}u)^{2}}-(d-1)]\frac{1}{u}\leq
0$ for $t<t^*$, i.e. $u(x,t)$ is decreasing in time $t$. This
together with (~\ref{eq:nonzero}) implies that
$\displaystyle\lim_{t\rightarrow t^*}u(x,t)$ exists and $>0$ for any $x\not=0.$
This proves Statement (4) and with it completes the proof of Theorem
~\ref{maintheorem}.
\begin{flushright}
$\square$
\end{flushright}


\section{Lyapunov-Schmidt Splitting (Effective Equations)}
\label{Section:Splitting}
Lemma \ref{induction} and Equation (~\ref{eq:definew}) imply that
there is a time $0<t_{\#}\leq \infty$ such that the solution
$w(y,\tau)=v(y,\tau) e^{-\frac{a}{4} y^2}$ of \eqref{eqn:w} can be
decomposed as:
\begin{equation}
w=w_{ab}+\xi,\ \xi\bot\ \phi_{0,a},\ \phi_{2 a}, \label{eqn:split3}
\end{equation}
with the functions
$\phi_{0,a}:=(\frac{a}{2\pi})^\frac{1}{4}e^{-\frac{ay^{2}}{4}},\
\phi_{2,a}:=(\frac{a}{8\pi})^\frac{1}{4}(1-ay^{2})e^{-\frac{ay^{2}}{4}},$
the parameters $a$, $b$ being $C^{1}$ functions of $t,$ $w_{a
b}:=V_{a b } e^{-\frac{a}{4} y^2}$, the fluctuation
$\xi:=e^{-\frac{ay^{2}}{4}}\phi$ and the orthogonality understood in
the $L^{2}$ norm. According to their definition in Section
~\ref{Section:Reparam} the parameters $a$, $b$ and $c$ depend on the
rescaled time $\tau$ through the original time $t$: $a(t(\tau))$,
$b(t(\tau))$ and $c(t(\tau))$. To simplify the notation we will
write $a(\tau)$, $b(\tau)$ and $c(\tau)$ for $a(t(\tau))$,
$b(t(\tau))$ and $c(t(\tau))$. This will not cause confusion as the
original parameter functions $a(t)$, $b(t)$ and $c(t)$ are not used
in what follows. In this section we derive equations for the
parameters functions $a(\tau)$, $b(\tau)$ and $c(\tau)$ and the
fluctuation $\xi(y,\tau)$.


Substitute (~\ref{eqn:split3}) into (~\ref{eqn:w}) to obtain the
following equation for $\xi$
\begin{equation}\label{eq:xi}
\partial_{\tau}\xi(y,\tau)=-L(a,b)\xi+F(a,b)+N_{1}(a,b,\xi)+N_{2}(a,b,\xi)
\end{equation}
where $L(a,b)$ is the linear operator given by
$$L(a,b):=-\partial_{y}^{2}+\frac{a^{2}+ \partial_{\tau}a }{4}y^{2}-\frac{3a}{2}
-\frac{(d-1)(\frac{1}{2}+a)}{2(d-1)+b y^{2}}$$ and the functions
$F(a,b)$, $N_{1}(a,b,\xi)$ and $N_{2}(a,b,\xi)$ are defined as
\begin{equation}\label{eq:scource}
F:=\frac{1}{2}\exp{-\frac{a
y^{2}}{4}}(\frac{2(d-1)+by^{2}}{a+\frac{1}{2}})^{\frac{1}{2}}[\Gamma_{1}+\Gamma_{2}\frac{y^{2}}{2(d-1)+by^2}+F_{1}]
\end{equation}
with
\begin{equation}\label{eq:MainRemain}
\begin{array}{lll}
\Gamma_{1}&:=&\frac{ \partial_{\tau}a }{a+\frac{1}{2}}+a-\frac{1}{2}+\frac{b}{d-1};\\
& &\\
\Gamma_{2}&:=&- \partial_{\tau}b  -b(a-\frac{1}{2}+\frac{b}{d-1})-\frac{b^{2}}{d-1};\\
& &\\
F_{1}&:=&-\frac{1}{d-1}\frac{b^{3}y^{4}}{(2(d-1)+by^{2})^{2}};\\
& &\\
N_{1}(a,b,\xi)&:=&-\frac{d-1}{v}\frac{a+\frac{1}{2}}{2(d-1)+by^{2}}
\exp{\frac{ay^{2}}{4}} \xi^{2};\\
& &\\
N_{2}(a,b,\xi)&:=&-\exp{-\frac{a
y^{2}}{4}}\frac{(\partial_{y}v)^{2}\partial_{y}^{2}v}{1+(\partial_{y}v)^{2}}.
\end{array}
\end{equation} Here $v$ is the same as in (~\ref{eq:definev}) and is related to $\xi$ be (~\ref{eq:definew}) and (~\ref{eqn:split3}), and we ordered the terms in $F$
according to the leading power in $y^{2}$.

In the next three lemmas we prove estimates on the terms $N_{1},\
N_{2},\ \Gamma_{1},\ \Gamma_{2}$ and $F$. These estimates will be
used in later sections.
\begin{lemma}\label{LM:ESTnonlin}
Assume $v$ satisfies Condition (A) in Section ~\ref{SEC:ApriEst},
$v(y,0)\geq g(y,b_{0})$, $v(y,0)\in \langle y\rangle L^{\infty},$\\
$|\partial_{y}v(y,0)v^{-1/2}(y,0)|\leq \kappa_{0}\beta^{1/2}(0)$ and
$|\partial_{y}^{n}v(y,0)|\leq \kappa_{0}\beta^{\frac{n}{2}}(0),\
n=2,3,4$; and assume there exists a time $T\leq \tau(t_{\#})$ such
that for any $\tau\leq T,$ $|M(\tau)|$, $A(\tau), B(\tau)\leq
\beta^{-\frac{1}{4}}(\tau)$ and
$$ v(\cdot,\tau)\in \langle y\rangle L^{\infty},\
\partial_{y}v(\cdot,\tau),\ \partial_{y}^{2}v(\cdot,\tau)\in
L^{\infty}$$ and $v(y,\tau)\geq \frac{1}{4}\sqrt{2(d-1)}$. Then we
have
\begin{equation}\label{eq:ESTnonlin}
|N_{1}(a,b,\xi)|\lesssim \frac{1}{1+\beta y^{2}}
\exp{\frac{ay^{2}}{4}} |\xi|^{2},
\end{equation}
\begin{equation}\label{eq:estN1}
\|\exp{\frac{a y^{2}}{4}}N_{1}(a,b,\xi)\|_{m,n}\lesssim
\beta^{\frac{m+n+2}{2}}(\tau)P(M(\tau))
\end{equation} for $(m,n)=(3,0),\ (\frac{11}{10},0), \ (2,1);$
\begin{equation}\label{eq:mn12}
\|\exp{\frac{a y^{2}}{4}}N_{1}(a,b,\xi)\|_{1,2}\lesssim
\beta^{2}[M_{2,1}+M_{3,0}]+\beta^{5/2}M_{\frac{11}{10},0}^{2};
\end{equation}
\begin{equation}\label{eq:estN25}
\|\langle
y\rangle^{-5}\exp{\frac{ay^{2}}{4}}N_{2}(a,b,\xi)\|_{\infty}\lesssim
\beta^{3}[1+M_{1,2}+M_{2,1}^{2}+M_{1,2}M_{2,1}^{2}];
\end{equation} and, for $(m,n)=(3,0),\
(\frac{11}{10},0), \ (2,1),\ (1,2)$
\begin{equation}\label{eq:estN2}
\|\exp{\frac{a y^{2}}{4}}N_{2}(a,b,\xi)\|_{m,n}\lesssim
\beta^{\frac{m+n+2}{2}}(\tau)P(M(\tau)).
\end{equation}
\end{lemma}
\begin{proof}
By the explicit form of $N_{1}$ we have $$|N_{1}(a,b,\xi)|\lesssim
\frac{1}{|v|}\frac{\frac{1}{2}+a(\tau)}{1+b(\tau)y^{2}}
\exp{\frac{ay^{2}}{4}} \xi^{2}.$$

The assumptions $A(\tau),B(\tau)\leq \beta^{- \frac{1}{4} }(\tau)$
imply that $b=\beta+o(\beta)$ and $a=\frac{1}{2}+O(\beta)$ which
together with the assumption on $v$ yield (~\ref{eq:ESTnonlin}).

Now we prove (~\ref{eq:estN1}), $(m,n)=(3,0),\ (2,1).$  The estimate
of $(m,n)=(\frac{11}{10},0)$ is similar to that of $(3,0),$ and is
omitted.

We start with $\|\exp{\frac{a y^{2}}{4}}N_{1}(a,b,\xi)\|_{3,0}$.
Recall the definitions
$V_{ab}=(\frac{2(d-1)+by^{2}}{a+\frac{1}{2}})^{1/2}$ and
$\xi(y,\tau)=e^{-\frac{ay^{2}}{4}}\phi(y,\tau)$, the definition of
$N_{1}$ in (~\ref{eq:MainRemain}) yields
$$e^{\frac{ay^{2}}{4}}N_{1}=-\frac{d-1}{v}V_{ab}^{-2}\phi^{2}.$$ By
direct computations we obtain
$$\|\exp{\frac{a y^{2}}{4}}N_{1}(a,b,\xi)\|_{3,0}\leq
\|\phi\|_{3,0}\||V_{ab}|^{-\frac{11}{10}}\phi\|_{\infty}.$$ By the
definition of $V_{ab}$ we have $$|V_{ab}|^{-\frac{11}{10}}\lesssim
\langle y\rangle^{-\frac{11}{10}}\beta^{-\frac{11}{20}},$$ thus, recalling the definition (~\ref{eq:majorants}) of estimating functions,
$$\|\exp{\frac{a y^{2}}{4}}N_{1}(a,b,\xi)\|_{3,0}\lesssim \beta^{-\frac{11}{20}}\|\phi\|_{3,0}\|\phi\|_{\frac{11}{10},0}\leq \beta^{5/2}M_{3,0}M_{\frac{11}{10},0}.$$

Before estimating for the pairs $(1,2),\ (2,1)$, we recall the
decomposition of $v$ as $v(y,\tau)=V_{a(\tau),b(\tau)}+\phi(y,\tau)$
with the function $V_{a,b}$ admitting the estimates
\begin{equation}\label{eq:derivatives}
V^{-k}_{ab}\lesssim \beta^{-\frac{k}{2}}\langle y\rangle^{-k}\
\text{for any}\ k\geq 0,\ |\partial_{y}^{n}V_{a,b}^{-1}|\lesssim
\beta^{\frac{n}{2}}V_{a,b}^{-1},\
\|\partial_{y}^{n}V_{a,b}\|_{\infty}\lesssim \beta^{\frac{n}{2}},\
n=1,2,3,
\end{equation}
by the assumptions on the estimating functions $A$ and $B.$ Also
recall the inequalities $v\geq \frac{1}{4}\sqrt{2(d-1)}$ and
$|\frac{\partial_{y}v(y,\tau)}{v^{ \frac{1}{2} }(y,\tau)}|\leq
2\beta^{ \frac{1}{2} }(\tau),$ $|\partial_{y}^{n}v(y,\tau)|\lesssim
\beta^{\frac{n}{2}}(\tau),\ n=2,3,4$ proved in Proposition
~\ref{PROP:UppLow}.
By direct computation and the recalled facts above we have
$$
\begin{array}{lll}
|\partial_{y}\exp{\frac{a y^{2}}{4}}N_{1}(a,b,\xi)|&\leq&
|2V^{-3}\partial_{y}V||\phi^{2}|+|2V^{-2}\partial_{y}\phi||\phi|+|V^{-2}\phi^{2}||\frac{\partial_{y}v}{v^{-2}}|\\
&\lesssim&
|V^{-2}||\phi^{2}|\beta^{1/2}+|V^{-2}||\partial_{y}\phi||\phi|\\
&\lesssim &\beta^{2/5}|\langle
y\rangle^{-1/10}\phi|^{2}+\beta^{-\frac{11}{20}}\langle
y\rangle^{-\frac{11}{20}}|\partial_{y}\phi||\phi|
\end{array}
$$ consequently
$$
\begin{array}{lll} \|\exp{\frac{a
y^{2}}{4}}N_{1}(a,b,\xi)\|_{2,1}&\lesssim&
\beta^{-\frac{11}{20}}\|\phi\|_{2,1}\|\phi\|_{\frac{11}{10},0}+\beta^{2/5}\|\phi\|_{\frac{11}{10},0}^{2}\\
&\lesssim&
\beta^{5/2}[M_{2,1}M_{\frac{11}{10},0}+M^{2}_{\frac{11}{10},0}].
\end{array}
$$

The proof of (~\ref{eq:mn12}) is
more involved. By direct computation and the recalled facts in and
after (~\ref{eq:derivatives}) we have
$$
|\partial_{y}^{2}e^{\frac{a y^{2}}{4}}N_{1}(a,b,\xi)|\lesssim
V_{a,b}^{-2}[|{\partial_{y}^{2}\phi\phi}|+\beta^{ \frac{1}{2}
}|{\phi\partial_{y}\phi}|v^{-1}+|{(\partial_{y}\phi)^{2}}
{v^{-1}}|+\beta \phi^{2}].
$$
Again by the facts in and after (~\ref{eq:derivatives})
${|\partial_{y}\phi|}{v^{- \frac{1}{2} }}\leq {|\partial_{y}v|}{v^{-
\frac{1}{2} }}+{|\partial_{y}V_{ab}|}{v^{- \frac{1}{2} }}\lesssim
\beta^{ \frac{1}{2} }$ and $|\partial_{y}^{2}\phi|\leq
|\partial_{y}^{2}v|+|\partial_{y}^{2}V_{ab}|\lesssim \beta$, which
implies the estimate for the first two terms
$$\|\langle y\rangle^{-1}{\partial_{y}^{2}\phi\phi}
{V_{a,b}^{-2}}\|_{\infty}+\|\langle y\rangle^{-1}\beta^{ \frac{1}{2}
}{\phi\partial_{y}\phi}{V_{ab}^{-2}v^{-1}}\|_{\infty}\lesssim
\|\phi\|_{3,0}\leq \beta^{2}M_{3,0}.$$ Similarly for the third term
$$\|\langle y\rangle^{-1}{(\partial_{y}\phi)^{2}}
{v^{-1}V_{a,b}^{-2}}\|_{\infty}\lesssim \|\langle
y\rangle^{-1}V_{ab}^{-1}\partial_{y}\phi\|_{\infty}\|v^{-1}\partial_{y}\phi\|_{\infty}\lesssim\|\phi\|_{2,1}\leq
\beta^{2}M_{2,1}.$$ For the last term we have
$$\beta\|\langle y\rangle^{-1} {\phi^{2}}{
V_{a,b}^{-2}}\|_{\infty}\lesssim
\beta^{2/5}\|\phi\|_{\frac{11}{10},0}^{2}\leq
\beta^{5/2}M_{\frac{11}{10},0}^{2}.$$
Collecting the estimates above we have (~\ref{eq:mn12}).

Now we turn to (~\ref{eq:estN25}) and (~\ref{eq:estN2}). By the
definition of $N_{2}$ we have
$$|\exp{\frac{ay^{2}}{4}}N_{2}(a,b,\xi)|\lesssim
|\partial_{y}v|^{2}|\partial_{y}^{2}v|.$$ Recall that
$v=V_{a,b}+\phi$ with
$V_{a,b}:=(\frac{2(d-1)+by^{2}}{a+\frac{1}{2}})^{\frac{1}{2}}$ and
$\phi:=e^{\frac{ay^{2}}{4}}\xi$. The bounds
$|\partial_{y}^{2}V_{a,b}|, \ \langle
y\rangle^{-1}|\partial_{y}V_{a,b}|\lesssim \beta$ yield
$$|\langle y\rangle^{-2}\partial_{y}v|\leq \langle
y\rangle^{-2}|\partial_{y}V_{ab}|+\langle
y\rangle^{-2}|\partial_{y}\phi|\leq \beta+\beta^{2}M_{2,1}$$ and
$$|\langle y\rangle^{-1}\partial_{y}^{2}v|\leq \langle
y\rangle^{-1}|\partial_{y}^{2}V_{ab}|+\langle
y\rangle^{-1}|\partial_{y}^{2}\phi|\leq \beta+\beta^{2}M_{1,2}$$
which yield the estimate on $\langle
y\rangle^{-5}e^{\frac{ay^{2}}{4}}N_{2}(a,b,\xi)$ or Equation
(~\ref{eq:estN25}).

In what follows we only prove the cases of $(m,n)= (3,0), (1,2)$ of
(~\ref{eq:estN2}). The other cases are similar.

By the definition of $N_{2}$ we have
$$\|e^{\frac{ay^{2}}{4}}N_{2}(a,b,\xi)\|_{3,0}=\|\langle y\rangle^{-3}\frac{\partial_{y}^{2}v(\partial_{y}v)^{2}}{1+(\partial_{y}v)^{2}}\|\lesssim
\|\langle
x\rangle^{-2}\partial_{y}v\|_{\infty}^{3/2}\|\partial_{y}^{2}v\|_{\infty}\lesssim
\beta^{5/2}(1+M_{2,1}^{2})$$ where recall the facts $
|\partial_{y}^{n}v(y,\tau)|\lesssim \beta^{\frac{n}{2}}(\tau),\
n=2,3,4$ proved in Proposition ~\ref{PROP:UppLow}, and recall the
definition of $v$ as $v(y,\tau)=V_{a,b}+\phi(y,\tau)$ with $V_{a,b}$
admitting the estimates
$$\|\partial_{y}^{n}V_{a,b}\|_{\infty}\lesssim \beta^{\frac{n}{2}},\ n=1,2,3,\ \|\langle y\rangle^{-1}\partial_{y}V_{a,b}\|_{\infty}\lesssim \beta.$$

By direct computation we have
$$|\partial_{y}^{2}N_{2}(a,b,\xi)|\lesssim |\partial_{y}^{4}v||\frac{\partial_{y}v}{1+(\partial_{y}v)^{2}}|+|\partial_{y}^{3}v||\partial_{y}^{2}v|+|(\partial_{y}^{2}v)^{3}|$$Using the estimates on
$\partial_{y}^{n}v$ listed above and the estimate $\|\langle
y\rangle^{-2}\partial_{y}v\|_{\infty}\lesssim
\beta+\beta^{2}M_{2,1}$ we have
$$
\begin{array}{lll}
\|N_{2} (a,b,\xi) \|_{1,2}&\lesssim& \|v\|_{0,4}\|v\|_{2,1}^{ \frac{1}{2} }+\|v\|_{0,3}\|v\|_{0,2}+\|v\|_{0,2}^{3}\\
&\lesssim& \beta^{5/2}(1+M_{2,1}).
\end{array}
$$ Thus the proof is complete.
\end{proof}
Now we establish some estimates for $\Gamma_{1}$ and $\Gamma_{2}$
defined after Equation (~\ref{eq:scource}).
\begin{lemma}
If $|M(\tau)|,\ A(\tau),\ B(\tau)\leq \beta^{- \frac{1}{4} }(\tau)$
and $v(y,\tau)\geq \frac{1}{4}\sqrt{2(d-1)}$ then we have
\begin{equation}\label{eq:estI1I2}
|\Gamma_{1}|,\ |\Gamma_{2}| \lesssim
\beta^{3}P(M,A)
\end{equation} where $P(M,A)$ is a nondecreasing polynomial of the vector $M$ and $A.$
\end{lemma}
\begin{proof}
Taking the inner products of (~\ref{eq:xi}) with the functions
$\phi_{0,a}:=\exp{-\frac{a y^{2}}{4}}$ and $
\phi_{2,a}:=(ay^{2}-1)\exp{-\frac{a y^{2}}{4}}$ and using the
orthogonality conditions in (~\ref{eqn:split2}) we have
\begin{equation}\label{eq:Projection}
\begin{array}{lll}
|\langle F(a,b),\phi_{0,a}\rangle|&\leq & G_{1}\\
& &\\
|\langle F(a,b),\phi_{2,a}\rangle|&\leq & G_{2}
\end{array}
\end{equation} where the functions $G_{1},\ G_{2}$ are defined as
$$G_{1}:=\langle |by^{2}\xi|+| \partial_{\tau}a y^{2}\xi|+|N_{1}(a,b,\xi)|+|N_{2}(a,b,\xi)|,\exp{-\frac{a}{4}y^{2}}\rangle$$
and $$G_{2}:=\langle |by^{2}\xi|+| \partial_{\tau}a
y^{2}\xi|+|N_{1}(a,b,\xi)|+|N_{2}(a,b,\xi)|,|(ay^{2}-1)|\exp{-\frac{a}{4}y^{2}}\rangle.$$
By Equations (~\ref{eq:ESTnonlin}), (~\ref{eq:estN25}) and the
assumptions on $A(\tau)$ and $B(\tau)$ we have that
$$G_{1},G_{2}\lesssim | \partial_{\tau}a |\beta^{2}M_{3,0}+\beta^{3}[1+M_{1,2}+M_{2,1}^{2}+M_{1,2}M_{2,1}^{2}+M_{3,0}^{2}].$$

We rewrite the function $F(a,b)$ in (~\ref{eq:scource}) as
$$F(a,b)=\chi(a,b)[\Gamma_{1}+\Gamma_{2}\frac{1}{a[2(d-1)+by^{2}]}+\Gamma_{2}\frac{ay^{2}-1}{a[2(d-1)+by^{2}]}+F_{1}].$$
where, recall, $F_{1}$ is defined in (~\ref{eq:MainRemain}),
and $\chi(a,b)$ is defined as
$$\chi(a,b):=(\frac{2(d-1)+by^{2}}{a+\frac{1}{2}})^{\frac{1}{2}}\frac{1}{2}\exp{-\frac{a y^{2}}{4}}.$$
By using the fact that $\phi_{0,a}\perp\phi_{2,a}$ we have
$$|\langle F(a,b),\phi_{0,a}\rangle|\gtrsim |\Gamma_{1}+\frac{1}{
{2a(d-1)}}\Gamma_{2}|-|b|(|\Gamma_{1}|+|\Gamma_{2}|)-|b|^{3}$$ and
$$|\langle F(a,b),\phi_{2,a}\rangle|\gtrsim
|\Gamma_{2}|-|b|(|\Gamma_{1}|+|\Gamma_{2}|)-b^{3},$$ which together
with (~\ref{eq:Projection}) implies that
$$
|\Gamma_{1}|,\ |\Gamma_{2}|\lesssim
\beta^{3}[1+M_{1,2}+M_{2,1}^{2}+M_{1,2}M_{2,1}^{2}+M_{3,0}^{2}]+|
\partial_{\tau}a |\beta^{2}M_{3,0}.
$$
Moreover by the definition of $\Gamma_{1}$ $ \partial_{\tau}a $ has
the bound
$$| \partial_{\tau}a |\lesssim |\Gamma_{1}|+\beta^{2}A.$$
Consequently $$|\Gamma_{1}|,\ |\Gamma_{2}|\lesssim
\beta^{3}[1+M_{1,2}+M_{2,1}^{2}+M_{1,2}M_{2,1}^{2}+M_{3,0}^{2}]+|\Gamma_{1}|\beta^{2}M_{3,0}+\beta^{4}AM_{3,0}.$$
This together with the assumption that $|M(\tau)|\leq \beta^{-
\frac{1}{4} }(\tau)$ implies (~\ref{eq:estI1I2}).
\end{proof}
To facilitate the later estimates we list the estimates of $F$ in
the following lemma.
\begin{lemma}
If $A(\tau),\ B(\tau)\leq \beta^{- \frac{1}{4} }(\tau)$ and
$(m,n)=(3,0), (\frac{11}{10},0),\ (1,2), \ (2,1),$ then
\begin{equation}\label{eq:estSource}
\| \exp{\frac{ay^{2}}{4}} F(a,b)\|_{m,n} \lesssim
\beta^{\frac{m+n+2}{2}}(\tau)P(M,A).
\end{equation}
\end{lemma}
\begin{proof}
In the following we only prove the cases $(m,n)=(\frac{11}{10},0),\
(3,0)$. The proof of the remaining cases are similar.

We start with $(m,n)=(\frac{11}{10},0).$ Recall the definition of
$F$ in Equation (~\ref{eq:scource}). We observe that $$\|\langle
y\rangle^{-\frac{11}{10}}\frac{y^{2}}{(2(d-1)+by^{2})^{ \frac{1}{2}
}}\|_{\infty}\leq b^{-\frac{9}{20}}\lesssim \beta^{-\frac{9}{20}}.$$
Thus the estimates on $\Gamma_{1}$ and $\Gamma_{2}$ in
(~\ref{eq:estI1I2}) imply that
$$
\begin{array}{lll}
\|\langle
y\rangle^{-\frac{11}{10}}(\frac{2(d-1)+by^{2}}{a+\frac{1}{2}})^{\frac{1}{2}}[\Gamma_{1}+\Gamma_{2}\frac{y^{2}}{2(d-1)+by^2}]\|_{\infty}
&\lesssim & (|\Gamma_{1}|+|\Gamma_{2}|)\beta^{-\frac{9}{20}}\lesssim
\beta^{ \frac{51}{20} }P(M,A).
\end{array}
$$ For the term $F_{1}$ by similar reasoning we have
$$\|\langle y\rangle^{-\frac{11}{10}}F_{1}(\frac{2(d-1)+by^{2}}{a+\frac{1}{2}})^{\frac{1}{2}}\|_{\infty}\lesssim
\beta^{ \frac{31}{20} }.$$ Combining the estimates above we complete
the estimate for $(m,n)=(\frac{11}{10},0).$

The estimate for $(m,n)=(3,0)$ is easier by using the observation
$$
\begin{array}{lll}
\|\langle
y\rangle^{-3}(\frac{2(d-1)+by^{2}}{a+\frac{1}{2}})^{\frac{1}{2}}[\Gamma_{1}+\Gamma_{2}\frac{y^{2}}{2(d-1)+by^2}]\|_{\infty}
&\lesssim & |\Gamma_{1}|+|\Gamma_{2}| \lesssim \beta^{3}P(M,A).
\end{array}
$$ and
$$\|\langle
y\rangle^{-3}F_{1}(\frac{2(d-1)+by^{2}}{a+\frac{1}{2}})^{\frac{1}{2}}\|_{\infty}\lesssim
\beta^{5/2}.$$

Collecting the estimates above we complete the proof.
\end{proof}
\section{Proof of Estimates
\eqref{eq:B}-\eqref{eq:A}}\label{SEC:EstB} The following lemmas show
that $b$ and $\beta$ are closely related.
\begin{lemma}
If $B(\tau)\le \beta^{-\frac{1}{4}}(\tau)$ for $\tau\in[0,T]$, then
(~\ref{eq:B}) holds.
\end{lemma}
\begin{proof}
We begin with rewriting equation \eqref{eq:estI1I2} as
\begin{equation*}
| \partial_{\tau}b  + \frac{1}{d-1}b^2|\lesssim b|
\frac{1}{2}-a+\frac{1}{d-1}b |+\beta^{3}P(M,A).
\end{equation*}
The first term on the right hand side is bounded by $b \beta^{2}
A\lesssim \beta^{3}A$ by the definition of $A$.
Hence we have
\begin{align}
\left| \partial_{\tau}b  + \frac{1}{d-1}b^2\right|&\lesssim \beta^3
P(M,A) \label{eqn:Simplifiedb},
\end{align}
To prove (~\ref{eq:B}) we begin by dividing
\eqref{eqn:Simplifiedb} by $b^2$ and using the inequality $\beta\lesssim
b$ to obtain the estimate
\begin{equation}
\left| \partial_\tau\frac{1}{b}-\frac{1}{d-1} \right|\lesssim \beta
P(M,A). \label{est:InversebDE}
\end{equation}
Since $\beta$ satisfies $-\partial_\tau \beta^{-1}+\frac{1}{d-1}=0$,
Equation \eqref{est:InversebDE} implies that
\begin{equation*}
|\partial_\tau( \frac{1}{b}-\frac{1}{\beta})|\lesssim
\beta P(M,A).
\end{equation*}
Integrating this equation over $[0,\tau]$, multiplying the result by
$\beta^{\frac{1}{4}}$ and using that $b\lesssim \beta$ and
$\beta(0)=b(0)$ give the estimate
\begin{equation*}
\beta^{-7/4}|\beta-b|\lesssim \beta^{\frac{1}{4}}\int_0^\tau \beta
P(M,A) ds
\end{equation*} which together with definitions of $\beta$ and $B$
implies (~\ref{eq:B}).
\end{proof}
\begin{lemma}
If $A(\tau),\ B(\tau)\le \beta^{-\frac{1}{4}}(\tau)$ for
$\tau\in[0,T]$, then (~\ref{eq:A}) holds.
\end{lemma}
\begin{proof}
Define the quantity $\Gamma:=a-\frac{1}{2}+\frac{1}{d-1} b$.  We
prove the proposition by integrating a differential inequality for
$\Gamma$. Differentiating $\Gamma$ with respect to $\tau$ and
substituting for $ \partial_{\tau}b  $ and $\partial_\tau a$ the
expression in terms of $\Gamma_{1}$ and $\Gamma_{2}$ (see
~\ref{eq:estI1I2}) and using Equation \eqref{eq:estI1I2}, we obtain
\begin{equation*}
\partial_\tau \Gamma+ ( a+\frac{1}{2}+\frac{1}{d-1}b
)\Gamma=- \frac{1}{(d-1)^{2}}b^2+{\cal R}_b.
\end{equation*} where ${\cal R}_{b}$ has the bound $$|\mathcal{R}_{b}|\leq
\beta^{3}P(M,A).$$ Let $\mu=\exp{\int_{0}^{\tau} a(s)+\frac{1}{2}+
\frac{1}{d-1}b(s) ds}$. Then the above equation implies that
\begin{equation*}
\mu\Gamma=\Gamma_{0}-\int_0^\tau \frac{\mu b^2}{(d-1)^{2}}\,
ds+\int_0^\tau \mu {\cal R}_b\, ds.
\end{equation*}
We now use the inequality $b\lesssim \beta$ and the estimate of
${\cal R}_b$ to estimate over $[0,\tau]\le [0,T]$:
\begin{equation*}
\begin{array}{lll}
|\Gamma|\lesssim \mu^{-1}\Gamma_0+\mu^{-1}\int_0^\tau \mu \beta^2\,
ds+\mu^{-1}\int_0^\tau \mu \beta^{3} ds P(M,A).
\end{array}
\end{equation*}
For our purpose, it is sufficient to use the less sharp inequality
\begin{equation*}
|\Gamma|\lesssim\mu^{-1}\Gamma(0)+\mu^{-1}\int_0^\tau \mu \beta^2\,
ds [1+\beta(0) P(M,A)].
\end{equation*} The assumption that $A(\tau),\ B(\tau)\leq
\beta^{-\frac{1}{4}}(\tau)$ implies that
$a+\frac{1}{2}+\frac{1}{d-1}b\geq \frac{1}{2}.$ Thus, it is not
difficult to show that $\beta^{-2}\mu^{-1}\Gamma(0)\leq A(0)$ and
$\beta^{-2}\mu^{-1}\int_0^\tau \mu \beta^2\, ds$ are bounded and
hence
$$A\lesssim A(0)+1+\beta(0)P(M,A)$$ which is (~\ref{eq:A}).
\end{proof}

\section{Rescaling of Fluctuations on a Fixed Time Interval}\label{SEC:Rescale}
We return to our key equation (~\ref{eq:xi}). In this section we
re-parameterize the unknown function $\xi(y,\tau)$ in such a way
that the $y^{2}$-term in the linear part of the new equation has a
time-independent coefficient (cf ~\cite{DGSW}).

Let $t(\tau)$ be the inverse function to $\tau(t)$, where
$\tau(t)=\int_{0}^{t}\lambda^{-2}(s)ds$ for any $\tau\geq 0.$ Pick
$T>0$ and approximate $\lambda(t(\tau))$ on the interval $0\leq
\tau\leq T$ by the new trajectory, $\lambda_{1}(t(\tau)),$ tangent
to $\lambda(t(\tau))$ at the point $\tau=T:$
$\lambda_{1}(t(T))=\lambda(t(T)),$ and
$\alpha:=-\lambda_{1}(t(\tau))\partial_{t}\lambda_{1}(t(\tau))=a(T)$
where, recall
$a(\tau):=-\lambda(t(\tau))\partial_{t}\lambda(t(\tau)).$ Now we
introduce the new independent variables $z$ and $\sigma$ as
$z(x,t):=\lambda^{-1}_{1}(t)x$ and
$\sigma(t):=\int_{0}^{t}\lambda^{-2}_{1}(s)ds$ and the new unknown
function $\eta(z,\sigma)$ as
\begin{equation}\label{NewFun}
\lambda_{1}(t)\exp{\frac{\alpha}{4}z^{2}}\eta(z,\sigma):=\lambda(t)\exp{\frac{a(\tau)}{4}y^{2}}\xi(y,\tau).
\end{equation}
In this relation one has to think of the variables $z$ and $y$,
$\sigma$, $\tau$ and $t$ as related by
$z=\frac{\lambda(t)}{\lambda_{1}(t)}y,$
$\sigma(t):=\int_{0}^{t}\lambda_{1}^{-2}(s)ds$ and
$\tau=\int_{0}^{t}\lambda^{-2}(s)ds,$ and moreover
$a(\tau)=-\lambda(t(\tau))\partial_{t}\lambda(t(\tau))$ and
$\alpha=a(T).$

For any $\tau=\int_{0}^{t(\tau)}\lambda^{-2}(s)ds$ with $t(\tau)\leq
t(T)$ (or equivalently $\tau\leq T$) we define a new function
$\sigma(\tau):=\int_{0}^{t(\tau)}\lambda_{1}^{-2}(s)ds.$ Observe the
function $\sigma$ is invertible, we denote by $\tau(\sigma)$ as its
inverse. We define
\begin{equation}\label{T2}
S:=\int_{0}^{t(T)}\lambda_{1}^{-2}(s)ds.
\end{equation}
The new function $\eta$ satisfies the equation
\begin{equation}\label{eq:eta}
\frac{d}{d\sigma}\eta(\sigma)=-\mathcal{L}_{\alpha}\eta(\sigma)+\mathcal{W}(a,b)\eta(\sigma)+\mathcal{F}(a,b)(\sigma)+\mathcal{N}_{1}(a,b,\eta)+\mathcal{N}_{2}(a,b,\eta)
\end{equation}
with the operators
$$\mathcal{L}_{\alpha}:= L_{\alpha}+V$$
where
$$ L_{\alpha}:=-\partial_{z}^{2}+\frac{\alpha^{2}}{4}z^{2}-\frac{3\alpha}{2},$$
$$V:=-\frac{2 (d-1) \alpha}{2(d-1)+\beta(\tau(\sigma)) z^{2}},$$ and
$$\mathcal{W}(a,b):=-\frac{\lambda^{2}}{\lambda_{1}^{2}}\frac{(d-1)(a+\frac{1}{2})}{2(d-1)+b(\tau(\sigma))y^{2}}+\frac{2(d-1)\alpha}{2(d-1)+\beta(\tau(\sigma)) z^{2}},$$
with the function
$$\mathcal{F}(a,b):= \exp{-\frac{\alpha}{4}z^{2}}\exp{\frac{a}{4}y^{2}}\frac{\lambda_{1}}{\lambda}F,$$
and with the nonlinear terms
\begin{equation}\label{eq:DefNonlin}
\begin{array}{lll}
\mathcal{N}_{1}(a,b,\eta)
&:=&\frac{\lambda_{1}}{\lambda}\exp{-\frac{\alpha}{4}z^{2}}e^{\frac{a}{4} y^{2}} N_{1}(a,b,\xi),\\
& &\\
\mathcal{N}_{2}(a,b,\eta)&:=&\frac{\lambda_{1}}{\lambda}\exp{-\frac{\alpha}{4}z^{2}}e^{\frac{a}{4}
y^{2}} N_{2}(a,b,\xi)
\end{array}
\end{equation}
where, recall $F$, $N_{1}$ and $N_{2}$ are defined after
(~\ref{eq:scource}) and where $\tau$ and $y$ are expressed in terms
of $\sigma$ and $z.$ In the next proposition we prove that the new
trajectory is a good approximation of the old one.
\begin{proposition}\label{NewTrajectory}
For any $\tau\leq T$ we have that if $A(\tau)\leq \beta^{-
\frac{1}{4} }(\tau)$ then
\begin{equation}\label{eq:compare}
|\frac{\lambda}{\lambda_{1}}(t(\tau))-1|\lesssim \beta(\tau)
\end{equation}
for some constant $c$ independent of $\tau$.
\end{proposition}
\begin{proof} By the properties of $\lambda$ and $\lambda_{1}$ we have
\begin{equation}\label{EstLambda}
\partial_{\tau}[\frac{\lambda}{\lambda_{1}}(t(\tau))-1]
=2a(\tau)(\frac{\lambda}{\lambda_{1}}(t(\tau))-1)+G(\tau)
\end{equation} with
$$G:=\alpha-a+(\alpha-a)(\frac{\lambda}{\lambda_{1}}-1)[(\frac{\lambda}{\lambda_{1}})^{2}+
\frac{\lambda}{\lambda_{1}}+1]+a(\frac{\lambda}{\lambda_{1}}-1)^{2}[\frac{\lambda}{\lambda_{1}}+2].$$
By the definition of $A(\tau)$, if $A(\tau)\leq \beta^{- \frac{1}{4}
}(\tau)$ then
\begin{equation}\label{CauchA}
|a(\tau)-\alpha|,\ |a(\tau)-\frac{1}{2}|\lesssim \beta(\tau)
\end{equation} in the time interval
$\tau\in [0,T]$. Thus
\begin{equation}\label{Rem}
|G|\lesssim
\beta+(\frac{\lambda}{\lambda_{1}}-1)^{2}+|\frac{\lambda}{\lambda_{1}}-1|^{3}+\beta|\frac{\lambda}{\lambda_{1}}-1|.
\end{equation}
Observe that $\frac{\lambda}{\lambda_{1}}(t(\tau))-1=0$ when
$\tau=T.$ Thus Equations (~\ref{EstLambda}) can be rewritten as
\begin{equation}\label{LambdaRe}
\frac{\lambda_{1}}{\lambda}(t(\tau))-1=-\int_{\tau}^{T}e^{-\int^{s}_{\tau}2a(t)dt}G(s)ds.
\end{equation}
We claim that Equations (~\ref{CauchA}) and (~\ref{Rem}) imply
(~\ref{eq:compare}). Indeed, define an estimating function
$\Lambda(\tau)$ as
$$\Lambda(\tau):=\sup_{\tau\leq s\leq T}\beta^{-1}(s)|\frac{\lambda}{\lambda_{1}}(t(s))-1|.$$
Then (~\ref{LambdaRe}) and the assumption $A(\tau),\ B(\tau)\leq
\beta^{- \frac{1}{4} }(\tau)$ imply $2a\geq  \frac{1}{2} $ and
$$
\begin{array}{lll}
|\frac{\lambda}{\lambda_{1}}(t(\tau))-1| &\lesssim &
\int_{\tau}^{T}e^{-\frac{1}{2}(T-\tau)}[\beta(s)+\beta^{2}(s)\Lambda^{2}(\tau)+\beta^{2}(s)\Lambda(\tau)]ds\\
&\lesssim
&\beta(\tau)+\beta^{2}(\tau)\Lambda^{2}(\tau)+\beta^{3}(\tau)\Lambda^{3}(\tau)+\beta^{2}(\tau)\Lambda(\tau),
\end{array}
$$
or equivalently
$$\beta^{-1}(\tau)|\frac{\lambda}{\lambda_{1}}(t(\tau))-1|\lesssim
1+\beta(\tau)\Lambda^{2}(\tau)+\beta^{2}(\tau)\Lambda^{3}(\tau)+\beta(\tau)\Lambda(\tau).$$
Consequently by the fact that $\beta(\tau)$ and $\Lambda(\tau)$ are
decreasing functions we have $$\Lambda(\tau)\lesssim
1+\beta(\tau)\Lambda^{2}(\tau)+\beta^{2}(\tau)\Lambda^{3}(\tau)+\beta(\tau)\Lambda(\tau)$$
which together with $\Lambda(T)=0$ implies $\Lambda(\tau)\lesssim 1$
for any time $\tau\in [0,T].$ This estimate and the
definition of $\Lambda(\tau)$ imply (~\ref{eq:compare}).
\end{proof}

Now we prove lemmas which will be used in the proofs of (~\ref{eq:M30}) and (~\ref{eq:M20}). Recall the definition of the function
$v(y,\tau)$ in (~\ref{eq:definev}).
\begin{lemma}\label{Bridge} Assume all the conditions in Lemma
~\ref{LM:ESTnonlin}. Then for $\tau\leq T$ we have
\begin{equation}\label{eq:compareXiEta}
\|\exp{\frac{\alpha}{4}z^{2}}\eta(\cdot,\sigma)\|_{m,n}\lesssim
\beta^{\frac{m+n+1}{2}}(\tau(\sigma))M_{m,n}(\tau(\sigma));
\end{equation}
\begin{equation}\label{eq:nonlinearity}
\|\exp{\frac{\alpha}{4}z^{2}}\mathcal{N}_{2}(a,b,\eta)\|_{m,n}\lesssim
\beta^{\frac{m+n+2}{2}}(\tau(\sigma))P(M(T));
\end{equation}
\begin{equation}\label{eq:estF2}
\|\exp{\frac{\alpha}{4}z^{2}}\mathcal{F}(a,b)(\sigma)\|_{m,n}\lesssim
\beta^{\frac{m+n+2}{2}}(\tau(\sigma))P(M(T),A(T));
\end{equation}
\begin{equation}\label{est:termW}
\|\exp{\frac{\alpha}{4}z^{2}}\mathcal{W}\eta(\sigma)\|_{m,n}\lesssim
\beta^{\frac{m+n+2}{2}}(\tau(\sigma))P(M(T))
\end{equation} for $(m,n)=(3,0),\ (\frac{11}{10},0),\ (1,2),$ and
$(2,1)$;
\begin{equation}\label{eq:estN11}
\|\exp{\frac{\alpha}{4}z^{2}}N_{1}(a,b,\eta)\|_{m,n}\lesssim \beta^{\frac{m+n+2}{2}}P(M(T));
\end{equation} for $(m,n)=(3,0),\ (\frac{11}{10},0),\ (2,1)$;\begin{equation}\label{eq:estN12}
\|\exp{\frac{\alpha}{4}z^{2}}N_{1}(a,b,\eta)\|_{1,2}\lesssim
\beta^{2}[M_{3,0}(T)+M_{2,1}(T)]+\beta^{5/2}M_{\frac{11}{10},0}^{2}.
\end{equation}
\end{lemma}
\begin{proof}
In what follows we use implicitly that
\begin{equation}\label{eq:compare2}
\frac{\lambda_{1}}{\lambda}(t(\tau))-1=O(\beta(\tau)) \
\text{and therefore}\ \frac{\lambda_{1}}{\lambda}(t(\tau)),\
\frac{\lambda}{\lambda_{1}}(t(\tau))\leq 2
\end{equation} implied by (~\ref{eq:compare}), and from which $\langle z\rangle^{-n}\lesssim \langle y\rangle^{-n}$, $n=1,2,3$.

Recall that $e^{\frac{a}{4}y^{2}}\xi=\phi$ after
(~\ref{eqn:split3}). By the definitions of $\eta$ and $M_{m,n}$ in
(~\ref{NewFun}) and (~\ref{eq:majorants}) we have
$$\|\exp{\frac{\alpha}{4}z^{2}}\eta(\sigma)\|_{m,n}\lesssim \|\phi(\tau(\sigma))\|_{m,n}\lesssim
\beta^{\frac{m+n+1}{2}}(\tau(\sigma))M_{m,n}(\tau(\sigma))$$ which
is (~\ref{eq:compareXiEta}).

The relation between $\mathcal{N}_{2}$ and $N_{2}$ in
(~\ref{eq:DefNonlin}) and the estimates of $N_{2}$ in
(~\ref{eq:estN2}) imply
$$\|\exp{\frac{\alpha}{4}z^{2}}\mathcal{N}_{2}(a,b,\eta)\|_{m,n}\lesssim \|\exp{\frac{ay^{2}}{4}}N_{2}\|_{m,n}\lesssim \beta^{\frac{m+n+2}{2}}(\tau(\sigma))P(M(T))$$
which is (~\ref{eq:nonlinearity}). Similarly we prove
(~\ref{eq:estN11}) and (~\ref{eq:estN12}).

By the definition of $\mathcal{F}$ and the estimate of $F$ in
(~\ref{eq:estSource}) we have
$$
\begin{array}{lll}
\|\exp{\frac{\alpha}{4}z^{2}}\mathcal{F}(a,b)(\sigma)\|_{m,n}\lesssim
\|e^{\frac{a}{4}y^{2}}F(a,b)(\tau(\sigma))\|_{m,n}
\lesssim
\beta^{\frac{m+n+2}{2}}(\tau(\sigma))P(M(T),A(T))
\end{array}
$$ which is (~\ref{eq:estF2}).

Now we prove (~\ref{est:termW}). Equation (~\ref{eq:compareXiEta})
and the fact $y=\frac{\lambda(t)}{\lambda_{1}(t)}z$ after
(~\ref{NewFun}) yield
$$
\begin{array}{lll}
\|\exp{\frac{\alpha}{4}z^{2}}\mathcal{W}\eta(\sigma)\|_{m,n}\lesssim
[\Omega_{1}+\Omega_{2}]\|\exp{\frac{\alpha}{4}z^{2}}\eta(\sigma)\|_{m,n} \lesssim
\beta^{\frac{m+n+1}{2}}(\tau(\sigma))M_{m,n}(T)[\Omega_{1}+\Omega_{2}]
\end{array}
$$ with $$\Omega_{1}:=|\frac{\lambda}{\lambda_{1}}-1|+|a(\tau(\sigma))-\alpha|,\ \Omega_{2}:=|\frac{b(\tau(\sigma))-\beta(\tau(\sigma))}{\beta(\tau(\sigma))}|.$$
Equations (~\ref{eq:compare2}) and (~\ref{CauchA}) imply that
$\Omega_{1}\lesssim \beta;$ the assumption on $B$ and its definition
imply $\Omega_{2}\lesssim \beta^{ \frac{1}{2} }(\tau(\sigma)).$
Consequently
$$\Omega_{1}+\Omega_{2}\lesssim \beta^{ \frac{1}{2} }(\tau(\sigma)).$$ Collecting the
estimates above we have (~\ref{est:termW}).
\end{proof}

\begin{lemma}
For any $c_{1},c_{2}>0$ there exists a constant $c(c_{1},c_{2})$
such that
\begin{equation}\label{INT}
\int_{0}^{S}\exp{-c_{1}(S-s)}\beta^{c_{2}}(\tau(s))ds\leq
c(c_{1},c_{2})\beta^{c_{2}}(T).
\end{equation}
\end{lemma}
\begin{proof}
By the definition of $\tau(\sigma)$ we have that
$\sigma=\int_{0}^{t(\tau)}\lambda_{1}^{-2}(s)ds$ and
$\tau=\tau(\sigma)=\int_{0}^{t(\tau)}\lambda^{-2}(k)dk.$ By
(~\ref{eq:compare2}) we have
\begin{equation}\label{TauS1}
4\sigma\geq \tau(\sigma)\geq \frac{1}{4}\sigma
\end{equation} which
implies $\frac{1}{\frac{1}{b_{0}}+\frac{\tau(\sigma)}{d-1}}\leq
4\frac{1}{\frac{1}{b_{0}}+\frac{\sigma}{d-1}},$ which in turn gives
\begin{equation}\label{InI2}
\int_{0}^{S}\exp{-c_{1}(S-s)}\beta^{c_{2}}(\tau(s))ds\leq
c(c_{1},c_{2})\frac{1}{(\frac{1}{b_{0}}+\frac{S}{d-1})^{c_{2}}}.
\end{equation} Using
(~\ref{TauS1}) again we obtain $4S\geq \tau(S)=T\geq \frac{1}{4}S$
which together with (~\ref{InI2}) implies that
$$\int_{0}^{S}\exp{-c_{1}(S-s)}\beta^{c_{2}}(\tau(s))ds\leq
c(c_{1},c_{2})\frac{1}{(\frac{1}{b_{0}}+\frac{T}{d-1})^{c_{2}}}\lesssim
c(c_{1},c_{2})\beta^{c_{2}}(T).$$ Hence
$$\int_{0}^{S}\exp{-c_{1}(S-s)}\beta^{c_{2}}(\tau(s))ds\leq
c(c_{1},c_{2})\beta^{c_{2}}(T)$$ which is (~\ref{INT}).
\end{proof}

We consider the spectrum of the operator
$\mathcal{L}_{\alpha}$. Due to the quadratic term $\frac{1}{4}
\alpha z^2$, the operator $\mathcal{L}_{\alpha}$ has a
discrete spectrum. For $\beta z^{2}\ll 1$ it is closed to the
harmonic oscillator Hamiltonian
\begin{equation}
L_\alpha-\alpha:=-\partial_z^2+\frac{1}{4} \alpha^2
z^2-\frac{5\alpha}{2}.
\end{equation}
The spectrum of the operator $L_{\alpha}-\alpha$ is
\begin{equation}
\sigma({L}_\alpha-\alpha)=\left\{n \alpha |\
n=-2,-1,0,1,\ldots\right\}.
\end{equation}
Thus it is essential that we solve the evolution equation
(~\ref{eq:eta}) on the subspace orthogonal to the first three
eigenvectors of $L_{\alpha}.$ These eigenvectors, normalized, are
\begin{align}\label{eq:eigenvectors}
\phi_{0,\alpha}(z):=(\frac{\alpha}{2\pi})^\frac{1}{4}
\exp{-\frac{\alpha}{4}z^2},\
\phi_{1,\alpha}(z):=(\frac{\alpha}{2\pi})^{\frac{1}{4}}\sqrt{\alpha}z
\exp{-\frac{\alpha}{4}z^2},
\phi_{2,\alpha}(z):=(\frac{\alpha}{8\pi})^{\frac{1}{4}}(1-\alpha
z^2)\exp{-\frac{\alpha}{4}z^2}.
\end{align}
We define the orthogonal projection $\overline{P}^{\alpha}_{n}$ onto
the space spanned by the first $n$ eigenvectors of $L_{\alpha}$,
\begin{equation}\label{eq:projection}
\overline{P}^{\alpha}_{ n}=\displaystyle\sum_{m=0}^{n-1}|
\phi_{m,\alpha} \rangle \langle \phi_{m,\alpha} |
\end{equation} and the orthogonal projection $$P^{\alpha}_{n}:=1-\overline{P}^{\alpha}_{n}, n=1,2,3.$$

The following lemma establishes a relation between the functions
$\phi:=e^{\frac{ay^{2}}{4}}\xi$ and $\eta$ at the times $\sigma=S$,
$\tau=T$.
\begin{lemma} If $m+n\leq 3$ and $l\geq 0$ then
\begin{equation}\label{eq:AgreeEnd}
\langle z\rangle^{-l}\exp{\frac{\alpha
z^{2}}{4}}P^{\alpha}_{m}(\partial_{z}+\frac{\alpha}{2}z)^{n}\eta(z,S)=\langle
y\rangle^{-l}\partial^{n}_{y}\phi(y,T).
\end{equation}
\end{lemma}
\begin{proof}
By the various definitions in (~\ref{NewFun}) we have
$\lambda_{1}(t(T))=\lambda(t(T)),\ a(T)=\alpha,$ and hence $z=y$
and, $e^{\frac{\alpha}{4}z^{2}}\eta(z,S)=\phi(y,T).$ Thus by the
fact
$e^{\frac{\alpha}{4}z^{2}}(\partial_{z}+\frac{\alpha}{2}z)e^{-\frac{\alpha}{4}z^{2}}=\partial_{z}$
we have
\begin{equation}\label{eq:derivPro}
\langle z\rangle^{-l}\exp{\frac{\alpha
z^{2}}{4}}P^{\alpha}_{m}(\partial_{z}+\frac{\alpha}{2}z)^{n}\eta(z,S)=\langle
y\rangle^{-l}\exp{\frac{a(T) y^{2}}{4}}P^{a(T)}_{m}\exp{-\frac{a(T)
y^{2}}{4}}\partial_{y}^{n} \phi(y,T).
\end{equation} By a standard integrating by part technique, the condition
$\xi(\cdot,\tau)\perp \phi_{k,a(\tau)}, \ k=0,1,2,$ and the
definitions of $\phi_{k,a}$ above yield
$$
\exp{-\frac{a(\tau) y^{2}}{4}}\partial^{n}_{y}\phi(\cdot,\tau)\perp
\phi_{k,a(\tau)}, \ 0\leq k\leq 2-n,
$$ i.e. $P^{a(T)}_{m}\exp{-\frac{a(T)
y^{2}}{4}}\partial_{y}^{n} \phi(y,T)=\exp{-\frac{a(T)
y^{2}}{4}}\partial_{y}^{n} \phi(y,T)$ if $m+n\leq 3.$ This together
with (~\ref{eq:derivPro}) implies (~\ref{eq:AgreeEnd}).
\end{proof}

The following proposition provides various decay estimates on the
propagators generated by $- L_{\alpha}$ and
$-\mathcal{L}_{\alpha}$.
\begin{proposition}\label{PRO:propagator}
For any function $g$ and times $\tau,\ \sigma$ with $\tau\geq
\sigma\geq 0$ we have
\begin{equation}\label{eq:estproject2}
\|\langle z\rangle^{-n}\exp{\frac{\alpha}{4}z^{2}}\exp{-
L_{\alpha}\sigma} P^{\alpha}_{2} g\|_{\infty}\lesssim
\exp{(1-n)\alpha\sigma}\|\langle
z\rangle^{-n}\exp{\frac{\alpha}{4}z^{2}}g\|_{\infty}
\end{equation} with $2\geq  n\geq 1;$
and there exist constants $c_{0},\delta>0$ such that if
$\beta(0)\leq \delta$, then
\begin{equation}\label{eq:OperatorWithV}
\|\langle z\rangle^{-n}\exp{\frac{\alpha}{4}z^{2}} P^{\alpha}_{n}
U_{n}(\tau,\sigma) P^{\alpha}_{n} g\|_{\infty}\lesssim
\exp{-(c_{0}+(n-3)\alpha)(\tau-\sigma)}\|\langle
z\rangle^{-n}\exp{\frac{\alpha}{4}z^{2}}g\|_{\infty}
\end{equation} where $U_{n}(\tau,\sigma)$ denotes the propagator
generated by the operator $- P^{\alpha}_{n}
\mathcal{L}_{\alpha} P^{\alpha}_{n}$, $n=1,2,3.$
\end{proposition}
\begin{proof}
By results of ~\cite{BrKu,DGSW,GaSi} we have
\begin{equation}\label{eq:inter} \|\langle
z\rangle^{-n}e^{\frac{\alpha}{4}z^{2}}e^{-
L_{\alpha}\sigma}P_{n}^{\alpha}g\|_{\infty}\lesssim
e^{-\sigma(n-1)}\|\langle
z\rangle^{-n}e^{\frac{\alpha}{4}z^{2}}g\|_{\infty},\
n=1,2.\end{equation} In particular by the estimate $n=1$,
$$\|\langle z\rangle^{-1}e^{\frac{\alpha}{4}z^{2}}e^{-
L_{\alpha}\sigma}P_{2}^{\alpha}g\|_{\infty}\lesssim \|\langle
z\rangle^{-1}e^{\frac{\alpha}{4}z^{2}}P_{2}^{\alpha}g\|_{\infty}\lesssim\|\langle
z\rangle^{-1}e^{\frac{\alpha}{4}z^{2}}g\|_{\infty}.
$$ using the fact
$P_{1}^{\alpha}P_{2}^{\alpha}=P_{2}^{\alpha}$ and the explicit form
of $P_{2}^{\alpha}$. For $2\geq n\geq 1$ we use the interpolation
technique to get Equation (~\ref{eq:estproject2}).

The case $n=3$ of (~\ref{eq:OperatorWithV}) was proved in
~\cite{DGSW}, Proposition 10 (cf ~\cite{BrKu,GaSi}). The proof of
the other cases is similar, thus is omitted.
\end{proof}
\section{Estimate of $M_{3,0}$}\label{SEC:EstM1}
In this section we prove Estimate (~\ref{eq:M30}) on the function
$M_{3,0}$. Given any time $\tau$, choose $T=\tau.$ Then we have the
estimates of Proposition ~\ref{NewTrajectory} for $\tau\leq T$. We
start from estimating $\eta$ defined in Equation (~\ref{NewFun}). We
observe that the function $\eta$ is not orthogonal to the first
three eigenvectors of the operator $L_{\alpha}.$ Therefore we derive
the equation for $P^{\alpha}_{3}\eta$:
\begin{equation}\label{EQ:eta2}
\frac{d}{d\sigma} P^{\alpha}_{3} \eta=-
P^{\alpha}_{3} \mathcal{L}_{\alpha} P^{\alpha}_{3}
\eta+
\sum_{k=1}^{5}D^{(k)}_{3,0}(\sigma)
\end{equation}
where
the functions $D^{(k)}_{m,n}\equiv D^{(k)}_{m,n}(\sigma), \
k=1,2,3,4, 5$, $(m,n)=(3,0),\ (2,0),(1,2),\ (2,1),$ are defined as
$$D^{(1)}_{m,n}:=- P^{\alpha}_{m} V\exp{-\frac{\alpha}{4}z^{2}}\partial_{z}^{n} [\exp{\frac{\alpha}{4}z^{2}}\eta]+ P^{\alpha}_{m} V P^{\alpha}_{m} \exp{-\frac{\alpha}{4}z^{2}}\partial_{z}^{n} [\exp{\frac{\alpha}{4}z^{2}}\eta],$$
$$D^{(2)}_{m,n}:= P^{\alpha}_{m} \exp{-\frac{\alpha}{4}z^{2}}\partial_{z}^{n} [\exp{\frac{\alpha}{4}z^{2}}\mathcal{W}\eta],$$
$$D^{(3)}_{m,n}:= P^{\alpha}_{m} \exp{-\frac{\alpha}{4}z^{2}}\partial_{z}^{n}[\exp{\frac{\alpha}{4}z^{2}}\mathcal{F}(a,b)],$$
$$D^{(4)}_{m,n}:= P^{\alpha}_{m} \exp{-\frac{\alpha}{4}z^{2}}\partial_{z}^{n}[\exp{\frac{\alpha}{4}z^{2}}\mathcal{N}_{1}(a,b,\alpha,\eta)],$$
$$D^{(5)}_{m,n}:= P^{\alpha}_{m} \exp{-\frac{\alpha}{4}z^{2}}\partial_{z}^{n}[\exp{\frac{\alpha}{4}z^{2}}\mathcal{N}_{2}]$$
where, recall the definitions of the function $\mathcal{F}$, the
operator $\mathcal{W}$ after Equation (~\ref{eq:eta}) and the
definition of $P_{m}^\alpha$ before (~\ref{eq:projection}).

Now we start with estimating the terms $D^{(k)}_{3,0}$,
$k=1,2,3,4,5$, on the right hand side of (~\ref{EQ:eta2}).
\begin{lemma}\label{EstDs} If $ A(\tau),\ B(\tau)\leq
\beta^{- \frac{1}{4} }(\tau)$ and if $\sigma\leq S$ (equivalently
$\tau\leq T$) then we have
\begin{equation}\label{eq:MplusN3}
\sum_{k=1}^{5}\|\exp{\frac{\alpha}{4}z^{2}}P_{3}^{\alpha}D_{3,0}^{(k)}(\sigma)\|_{3,0}\lesssim
\beta^{\frac{5}{2}}(\tau(\sigma))P(M(T),A(T)).
\end{equation}
\end{lemma}
\begin{proof}
We rewrite $ P^{\alpha}_{3} D^{(1)}_{3,0},$ as
$$ P^{\alpha}_{3} D_{3,0}^{(1)}(\sigma)= P^{\alpha}_{3} \frac{\alpha+\frac{1}{2}}{2(d-1)+b(\tau(\sigma))z^{2}}b(\tau(\sigma))z^{2}(1- P^{\alpha}_{3} )\eta(\sigma)$$
which admits the estimate $$
\begin{array}{lll}
\|\exp{\frac{\alpha}{4}z^{2}} P^{\alpha}_{3}
D_{3,0}^{(1)}(\sigma)\|_{3,0}&\lesssim& |\langle
z\rangle^{-1}\frac{b(\tau(\sigma))z^{2}}{1+bz^{2}}|\| \exp{\frac{\alpha}{4}z^{2}}(1- P^{\alpha}_{3} )\eta(\sigma)\|_{2,0}\\
&\lesssim & b^{ \frac{1}{2} }(\tau(\sigma))\|\exp{\frac{\alpha}{4}z^{2}}\eta(\sigma)\|_{3,0}\\
&\lesssim & \beta^{\frac{m+n+2}{2}}(\tau(\sigma))M_{3,0}(T)
\end{array}
$$ where we use (~\ref{eq:compareXiEta}), the fact that $|b(\tau)|\leq 2\beta(\tau)$ implied by $B(\tau)\leq
\beta^{- \frac{1}{4} }(\tau)$ and the fact for any $m\geq 1$
\begin{equation}\label{eq:estPro}
\|\exp{\frac{\alpha}{4}z^{2}}(1- P^{\alpha}_{m} )g\|_{m-1,0}\lesssim
\|\exp{\frac{\alpha}{4}z^{2}}g\|_{m,0}
\end{equation} by (~\ref{eq:projection}). Thus we have the estimate for $D^{(1)}_{3,0}$.

Now we estimate $D^{(k)}_{3,0}, \ k=2,3,4,5.$ First by
(~\ref{eq:estPro}) we observe
$$\sum_{k=2}^{5}\|\exp{\frac{\alpha}{4}z^{2}}P_{3}^{\alpha}D_{3,0}^{(k)}(\sigma)\|_{3,0}\leq
\|\exp{\frac{\alpha}{4}z^{2}}\mathcal{F}(a,b)(\sigma)\|_{3,0}+\|\exp{\frac{\alpha}{4}z^{2}}\mathcal{W}\eta(\sigma)\|_{3,0}+\|\exp{\frac{\alpha}{4}z^{2}}\mathcal{N}_{1}\|_{3,0}+\|\exp{\frac{\alpha}{4}z^{2}}\mathcal{N}_{2}\|_{3,0}.$$
The estimates of $\mathcal{F},\ \mathcal{N}_{1},\ \mathcal{N}_{2}$
and $\mathcal{W}\eta$ in (~\ref{eq:nonlinearity})-(~\ref{eq:estN11})
imply
$$\sum_{k=2}^{5}\|\exp{\frac{\alpha}{4}z^{2}}P_{3}^{\alpha}D_{3,0}^{(k)}(\sigma)\|_{3,0}\lesssim
\beta^{\frac{5}{2}}(\tau(\sigma))P(M(T),A(T)).$$

Collecting the estimates above we complete the proof.
\end{proof}
Now we prove Equation (~\ref{eq:M30}). Let $S$ and $T$ be the same
as in Section ~\ref{SEC:Rescale}. By Duhamel principle we rewrite
Equation (~\ref{EQ:eta2}) as
\begin{equation}\label{eq:eta3}
 P^{\alpha}_{3} \eta(S)=P^{\alpha}_{3}U_{3}(S,0) P^{\alpha}_{3} \eta(0)+\displaystyle\sum_{n=1}^{5}\int_{0}^{S}P^{\alpha}_{3}U_{3}(S,\sigma) P^{\alpha}_{3} D_{3,0}^{(n)}(\sigma)d\sigma,
\end{equation} where, recall, $U_{3}(\tau,\sigma)$  is defined and estimated in
(~\ref{eq:OperatorWithV}), from which we obtain
\begin{equation}\label{eq:M1Ge}
\beta^{-2}(T)\|\exp{\frac{\alpha}{4}z^{2}} P^{\alpha}_{3} \eta(S)\|_{3,0}\\
\lesssim
\exp{-c_{0}S}\beta^{-2}(T)\|\exp{\frac{\alpha}{4}z^{2}}\eta(0)\|_{3,0}+\beta^{-2}(T)\displaystyle\sum_{k=1}^{5}\int_{0}^{S}\exp{-c_{0}(S-\sigma)}\|\exp{\frac{\alpha}{4}z^{2}}D^{(k)}_{3,0}(\sigma)\|_{3,0}d\sigma.
\end{equation}

Now we estimate each term on the right hand side. We begin with the
first term. By the slow decay of $\beta(\tau)$ and Equation
(~\ref{eq:compareXiEta}) we have
\begin{equation}\label{eq:estIni}
\begin{array}{lll}
\exp{-c_{0}S}\beta^{-2}(T)\|\exp{\frac{\alpha}{4}z^{2}}\eta(0)\|_{3,0} \lesssim
\beta^{-2}(0)\|\exp{\frac{\alpha}{4}z^{2}}\eta(0)\|_{3,0} \lesssim  M_{3,0}(0).
\end{array}
\end{equation}
For the second term we use the integral estimate (~\ref{INT}) and
the estimate of $D_{3,0}^{(k)}$ in Equation (~\ref{eq:MplusN3}) to
obtain
\begin{equation}\label{eq:estD30}
\displaystyle\sum_{k=1}^{5}\int_{0}^{S}\exp{-c_{0}(S-\sigma)}\|\exp{\frac{\alpha}{4}z^{2}}D^{(k)}_{3,0}(\sigma)\|_{3,0}d\sigma
\lesssim \beta^{5/2}(T)P(M(T),A(T)).
\end{equation}

By (~\ref{eq:AgreeEnd}) we have $\|\exp{\frac{\alpha}{4}z^{2}}
P^{\alpha}_{3} \eta(S)\|_{3,0}=\|
\phi(\cdot,T)\|_{3,0}$ which together with (~\ref{eq:estIni}) and
(~\ref{eq:estD30}) implies
$$
\begin{array}{lll}
& &\beta^{-2}(T)\| \phi(\cdot,T)\|_{3,0} \lesssim M_{3,0}(0)+\beta^{
\frac{1}{2} }(T)P(M(T),A(T))
\end{array}
$$ where $P$ is a nondecreasing polynomial.
By the definition of $M_{3,0}$ in (~\ref{eq:majorants}) we obtain
$$
\begin{array}{lll}
M_{3,0}(T)&\lesssim& M_{3,0}(0)+\beta^{ \frac{1}{2}
}(0)P(M(T),A(T)).
\end{array}
$$
Since $T$ is an arbitrary Equation (~\ref{eq:M30}) follows.
\begin{flushright}
$\square$
\end{flushright}
\section{Proof of Equation (~\ref{eq:M20})}\label{SEC:EstM2}
We derive an equation for $P^{\alpha}_{2} \eta(\sigma)$ from
Equation (~\ref{eq:eta}) as
\begin{equation}\label{eq:eta4}
\frac{d}{d\sigma}P^{\alpha}_{2} \eta=- L_{\alpha}P^{\alpha}_{2}
\eta-P^{\alpha}_{2}
V\eta+P_{2}^{\alpha}\sum_{k=2}^{5}D_{2,0}^{(k)}
\end{equation}
where the functions $D^{(k)}_{2,0}$ and the operator $L_{\alpha}$ are
defined after (~\ref{EQ:eta2}) and (~\ref{eq:eta}) respectively.
\begin{lemma}
If $ A(\tau),\ B(\tau)\leq \beta^{- \frac{1}{4} }(\tau)$, then
\begin{equation}\label{eq:estF3}
\|\exp{\frac{\alpha}{4}z^{2}}V\eta(\sigma)\|_{\frac{11}{10},0}\lesssim \beta^{\frac{21}{20}}(\tau(\sigma))M_{3,0}(T).
\end{equation}
\begin{equation}\label{eq:1.1}
\sum_{k=1}^{5}\|\exp{\frac{\alpha}{4}z^{2}}P_{2}^{\alpha}D_{2,0}^{(k)}(\sigma)\|_{\frac{11}{10},0}\lesssim
\beta^{\frac{31}{20}}(\tau)P(M(T),A(T))
\end{equation}
\end{lemma}
\begin{proof}
By the assumption on $B$ we have $\frac{1}{b}\lesssim
\frac{1}{\beta}$ hence
$$\langle z\rangle^{-\frac{11}{10}}\frac{1}{1+b(\tau(\sigma))z^{2}}\leq \langle z\rangle^{-\frac{11}{10}}{(1+b(\tau(\sigma))z^{2})^{-\frac{19}{20}}}\lesssim
{\beta^{-19/20}(\tau(\sigma))}\langle z\rangle^{-3}$$ which together
with the definition of $V$ after (~\ref{eq:eta}) and the estimate in
(~\ref{eq:compareXiEta}) yields
$$
\begin{array}{lll}
\|\exp{\frac{\alpha}{4}z^{2}}V\eta(\sigma)\|_{\frac{11}{10},0}&\lesssim&
\|\frac{1}{{1+b(\tau(\sigma))z^{2}}}\exp{\frac{\alpha}{4}z^{2}}\eta(\sigma)\|_{\frac{11}{10},0}\\
&\lesssim & \beta^{-\frac{19}{20}}(\tau(\sigma))\|\exp{\frac{\alpha}{4}z^{2}}\eta(\sigma)\|_{3,0}\\
&\lesssim &\beta^{\frac{21}{20}}(\tau(\sigma))M_{3,0}(T).
\end{array}
$$
This gives (~\ref{eq:estF3}). The proof of (~\ref{eq:1.1}) is almost the same to that of (~\ref{eq:MplusN3}) and, thus omitted.
\end{proof}
Rewrite (~\ref{eq:eta4}) to have
$$
P^{\alpha}_{2} \eta(S)=\exp{- L_{\alpha}S} P^{\alpha}_{2}
\eta(0)
+\int_{0}^{S}\exp{- L_{\alpha}(S-\sigma)} P^{\alpha}_{2}[-
V\eta+\displaystyle\sum_{k=2}^{5}D_{2,0}^{(k)}]d\sigma,
$$ where, recall the definition
of $S$ in (~\ref{T2}). By the propagator estimate of
$\exp{- L_{\alpha}\sigma} P^{\alpha}_{2} $ in
(~\ref{eq:estproject2}), we have
\begin{equation}\label{K123s}
\begin{array}{lll}
\|\exp{\frac{\alpha}{4}z^{2}} P^{\alpha}_{2}
\eta(S)\|_{\frac{11}{10},0}\lesssim
K_{0}+K_{1}+K_{2}
\end{array}
\end{equation} where the functions $K_{n}$'s are given by
$$K_{0}:=\exp{-\alpha  S}\|\exp{\frac{\alpha}{4}z^{2}}\eta(0)\|_{\frac{11}{10},0},\ \
K_{1}:=\int_{0}^{S}\exp{-\alpha
(S-\sigma)}\|\exp{\frac{\alpha}{4}z^{2}}V\eta(\sigma)\|_{\frac{11}{10},0}d\sigma,$$
$$K_{2}:=\displaystyle\sum_{k=2}^{5}\int_{0}^{S}\exp{-\alpha (S-\sigma)}\|\exp{\frac{\alpha}{4}z^{2}}D_{2,0}^{(k)}\|_{\frac{11}{10},0}d\sigma.$$
Next, we estimate $K_{n}$'s, $n=0,1,2.$
\begin{itemize}
\item[(K0)]
Equation (~\ref{eq:compareXiEta}) and the slow decay of $\beta$
yield
\begin{equation}\label{eq:estK0}
K_{0}\lesssim\beta^{\frac{21}{20}}(T)\beta^{-\frac{21}{20}}(0)\|\exp{\frac{\alpha}{4}z^{2}}\eta(0)\|_{\frac{11}{10},0}\lesssim
\beta^{\frac{21}{20}}(T)M_{\frac{11}{10},0}(0).
\end{equation}
\item[(K1)]  The estimate in
(~\ref{eq:estF3}) and the integral estimate in (~\ref{INT}) imply
\begin{equation}\label{EstK1}
K_{1}\lesssim \int_{0}^{S}\exp{-\alpha
(S-\sigma)}\beta^{\frac{21}{20}}(\tau(\sigma))d\sigma
M_{3,0}(T)\lesssim \beta^{\frac{21}{20}}(T) M_{3,0}(T).
\end{equation}
\item[(K2)]
The estimates of $D_{2,0}^{(k)}$, $k=2,3,4,5$, in Equation
(~\ref{eq:1.1}) yield the bound
\begin{equation}\label{eq:estK2}
\begin{array}{lll}
K_{2}&\lesssim & \int_{0}^{S}\exp{-\alpha (S-\sigma)}
 \beta^{\frac{21}{20}+\frac{1}{2}}(\tau(\sigma))d\sigma P(M(T),A(T))\\
 &\lesssim &\beta^{\frac{21}{20}+\frac{1}{2}}(T)P(M(T),A(T)).
\end{array}
\end{equation}

\end{itemize}
Collecting the estimates (~\ref{K123s})-(~\ref{eq:estK2}) we have
\begin{equation}\label{FinalStep}
\beta^{-\frac{21}{20}}(T)\|\exp{\frac{\alpha}{4}z^{2}}
P^{\alpha}_{2} \eta(\tau(S))\|_{\frac{11}{10},0} \lesssim
M_{\frac{11}{10},0}(0)+M_{3,0}(T)+\beta^{ \frac{1}{2}
}(0)P(M(T),A(T)).
\end{equation}
By Equation (~\ref{eq:AgreeEnd}) we have
$$
\beta^{-\frac{21}{20}}(T)\| \phi(T)\|_{\frac{11}{10},0}
=\beta^{-\frac{21}{20}}(T)\|\exp{\frac{\alpha z^{2}}{4}}
P^{\alpha}_{2} \eta(S)\|_{\frac{11}{10},0}
$$
which together with (~\ref{FinalStep}) and the definition of
$M_{\frac{11}{10},0}$ implies
$$
\begin{array}{lll}
M_{\frac{11}{10},0}(T)&\lesssim &
M_{\frac{11}{10},0}(0)+M_{3,0}(T)+\beta^{ \frac{1}{2}
}(0)P(M(T),A(T)).
\end{array}
$$
Since $T$ is an arbitrary time, the proof is complete.
\section{Proof of Equation (~\ref{eq:M21})}\label{SEC:EstM21}
By Equation (~\ref{eq:eta}) and the observation
\begin{equation}\label{eq:obser}
\exp{-\frac{\alpha
z^{2}}{4}}\partial_{z}[\exp{\frac{\alpha}{4}z^{2}}g]=(\partial_{z}+\frac{\alpha}{2}z)g
\end{equation}
for any function $g$, the function
$P^{\alpha}_{2}(\partial_{z}+\frac{\alpha}{2}z)\eta$ satisfies
\begin{equation}\label{eq:oneD}
\frac{d}{d\sigma}P^{\alpha}_{2}(\partial_{z}+\frac{\alpha}{2}z)\eta=
-P^{\alpha}_{2}(\mathcal{L}_{\alpha}+\alpha)P^{\alpha}_{2}
(\partial_{z}+\frac{\alpha}{2}z)\eta
+\sum_{k=1}^{5}D^{(k)}_{2,1}+D_{6}
\end{equation} with $D^{(k)}_{2,1}$
defined after (~\ref{EQ:eta2}) and
$$D_{6}:=-P^{\alpha}_{2}\eta\partial_{z}V.$$ Thus applying the
operator $\partial_{z}+\frac{\alpha}{2}z$ leads to the equation with
improved linear part.
\begin{lemma}
If $A(\tau),\ B(\tau)\leq \beta^{- \frac{1}{4} }(\tau)$, then we
have
\begin{equation}\label{eq:estD6}
\|\exp{\frac{\alpha}{4}z^{2}}D_{6}(\sigma)\|_{2,0}\lesssim
\beta^{2}(\tau(\sigma))M_{3,0}(T).
\end{equation}
\begin{equation}\label{eq:estD21}
\|e^{\frac{\alpha
z^{2}}{4}}\sum_{k=1}^{5}D_{2,1}^{(k)}(\sigma)\|_{2,0}\lesssim
\beta^{5/2}(\tau(\sigma))P(M(T),A(T)).
\end{equation}
\end{lemma}
The proofs are the same as those of (~\ref{eq:estF3}) and (~\ref{eq:MplusN3}) and, thus are omitted.

By the Duhamel principle we rewrite Equation (~\ref{eq:oneD}) as
$$
P^{\alpha}_{2}(\partial_{z}+\frac{\alpha}{2}z)\eta(S)=P^{\alpha}_{2}U_{2}(S,0)\exp{-\alpha
S} P^{\alpha}_{2} (\partial_{z}+\frac{\alpha}{2}z)\eta(0)+
\int_{0}^{S}P^{\alpha}_{2} U_{2}(S,\sigma) \exp{-\alpha
(S-\sigma)}P^{\alpha}_{2}[\displaystyle\sum_{n=1}^{5}
D_{2,1}^{(n)}+D_6]d\sigma,
$$ where $U_{2}$ is defined and estimated in
(~\ref{eq:OperatorWithV}), from which we have
\begin{equation}\label{eq:estEta}
\|\exp{\frac{\alpha}{4}z^{2}}P^{\alpha}_{2}(\partial_{z}+\frac{\alpha}{2}z)\eta(S)]\|_{2,0}\lesssim
Y_{1}+Y_{2}+Y_{3}
\end{equation} with $$Y_{1}:=\exp{-c_{0}S}\|
\exp{\frac{\alpha}{4}z^{2}}\eta(0)\|_{2,1};$$
$$Y_{2}:=\int_{0}^{S}\exp{-c_{0}(S-\sigma)}\displaystyle\sum_{k=1}^{5}\|\exp{\frac{\alpha
z^{2}}{4}}D_{2,1}^{(k)}(\sigma)\|_{2,0}d\sigma;$$
$$Y_{3}:=\int_{0}^{S}\exp{-c_{0}(S-\sigma)}\|\exp{\frac{\alpha
z^{2}}{4}} D_{6}(\sigma)\|_{2,0}d\sigma.$$

Next, we estimate $Y_{n}, \ n=1,2,3.$ By (~\ref{eq:estD21})
and the integral estimate (~\ref{INT}) we have
\begin{equation}\label{eq:Y2}
Y_{2}\lesssim
\int_{0}^{S}\exp{-c_{0}(S-\sigma)}\beta^{5/2}(\tau(\sigma))d\sigma
P(M(T),A(T))\lesssim \beta^{5/2}(T)P(M(T),A(T));
\end{equation} by
(~\ref{eq:estD6}). By similar reasoning,
\begin{equation}
Y_{3}\lesssim
\beta^{2}(T)M_{3,0}(T);
\end{equation} and by
(~\ref{eq:compareXiEta}) and the slow decay of $\beta$,
\begin{equation}\label{eq:Y1}
Y_{1}\lesssim \exp{-c_{0}S}\|\phi(\cdot,0)\|_{2,1}\lesssim
\beta^{2}(T)M_{2,1}(0).
\end{equation} Collecting the estimates (~\ref{eq:estEta})-(~\ref{eq:Y1}) we obtain $$\beta^{-2}(T)\|\exp{\frac{\alpha
z^{2}}{4}}P^{\alpha}_{2}(\partial_{z}+\frac{\alpha}{2}z)\eta(S)]\|_{2,0}\lesssim
M_{2,1}(0)+M_{3,0}(T)+\beta^{ \frac{1}{2} }(0)P(M(T),A(T)).$$
Moreover by Equation (~\ref{eq:AgreeEnd}) we have
$\|\exp{\frac{\alpha
z^{2}}{4}}P^{\alpha}_{2}(\partial_{z}+\frac{\alpha}{2}z)\eta(S)]\|_{2,0}=\|\phi(T)\|_{2,1}$.
Thus by the definition of $M_{2,1}$
$$M_{2,1}(T)\lesssim
M_{2,1}(0)+M_{3,0}(T)+\beta^{ \frac{1}{2} }(0)P(M(T),A(T))$$ which
together with fact that $T$ is arbitrary implies (~\ref{eq:M21}).
\section{Proof of Equation (~\ref{eq:M12})}\label{SEC:estM12}
By Equation (~\ref{eq:eta}) the function
$P^{\alpha}_{1}(\partial_{z}+\frac{\alpha}{2}z)^{2} \eta$ satisfies
the equation
\begin{equation}\label{eq:twoD}
\begin{array}{lll}
\frac{d}{d\sigma}P^{\alpha}_{1}(\partial_{z}+\frac{\alpha}{2}z)^{2}\eta&=&-P^{\alpha}_{1}(\mathcal{L}_{\alpha}+2\alpha)P^{\alpha}_{1}(\partial_{z}+\frac{\alpha}{2}z)^{2}\eta+P^{\alpha}_{1}\displaystyle\sum_{k=1}^{5}D_{1,2}^{(k)}+D_{7}
\end{array}
\end{equation} with $D^{(k)}_{1,2}$
defined after (~\ref{EQ:eta2}) and
$$D_{7}:=-P^{\alpha}_{1}\exp{-\frac{\alpha
z^{2}}{4}}[\exp{\frac{\alpha}{4}z^{2}}\eta\partial_{z}^{2}V+2\partial_{z}[\exp{\frac{\alpha}{4}z^{2}}\eta]\partial_{z}V].$$
\begin{lemma}
If $ A(\tau),\ B(\tau)\leq \beta^{- \frac{1}{4} }(\tau)$, then we
have
\begin{equation}\label{eq:estD7}
\|\exp{\frac{\alpha}{4}z^{2}}D_{7}(\sigma)\|_{1,0}\lesssim
\beta^{2}(\tau(\sigma))[M_{3,0}(T)+M_{2,1}(T)],
\end{equation}
\begin{equation}\label{eq:estD12}
\|e^{\frac{\alpha}{4}z^{2}}\sum_{k=1}^{5}D_{1,2}^{(k)}(\sigma)\|_{1,0}
\lesssim
\beta^{2}(\tau(\sigma))[M_{3,0}(T)+M_{2,1}(T)]+\beta^{5/2}(\tau(\sigma))P(M(T),A(T)).
\end{equation}
\end{lemma}
The proofs are almost the same as those of (~\ref{eq:estF3}) and (~\ref{eq:MplusN3}), thus omitted.

By Duhamel principle we rewrite Equation (~\ref{eq:twoD}) as
$$
\begin{array}{lll}
P^{\alpha}_{1}(\partial_{z}+\frac{\alpha}{2}z)^{2}\eta(S)&=&P^{\alpha}_{1}U_{1}(S,0)\exp{-2\alpha
S} P^{\alpha}_{1}
(\partial_{z}+\frac{\alpha}{2}z)^{2}\eta(0)\\
& &+\int_{0}^{S}P^{\alpha}_{1} U_{1}(S,\sigma)\exp{-2\alpha
(S-\sigma)} P^{\alpha}_{1}[ \displaystyle\sum_{n=1}^{5}
D_{1,2}^{(n)}+D_7]d\sigma
\end{array}
$$ where, recall $U_{1}(t,s)$ is defined and estimated in
(~\ref{eq:OperatorWithV}), from which we have
\begin{equation}\label{eq:Eta12}
\begin{array}{lll}
\|\exp{\frac{\alpha
z^{2}}{4}}P^{\alpha}_{1}(\partial_{z}+\frac{\alpha}{2}z)^{2}\eta(S)\|_{1,0}\lesssim
Z_{1}+Z_{2}
\end{array}
\end{equation} with $$Z_{1}:=\exp{-c_{0}S}\| \exp{\frac{\alpha}{4}z^{2}}\eta(0)\|_{1,2};$$
$$Z_{2}:=\int_{0}^{S}\exp{-c_{0}(S-\sigma)}\displaystyle\sum_{k=1}^{5}\|\exp{\frac{\alpha
z^{2}}{4}}D_{1,2}^{(k)}(\sigma)\|_{1,0}d\sigma+\int_{0}^{S}\exp{-c_{0}(S-\sigma)}\|\exp{\frac{\alpha z^{2}}{4}}
D_{7}(\sigma)\|_{1,0}d \sigma.$$
 By (~\ref{eq:estD7}), (~\ref{eq:estD12}) we have
\begin{equation}
\begin{array}{lll}
Z_{2}&\lesssim&
\int_{0}^{S}\exp{-c_{0}(S-\sigma)}\beta^{5/2}(\tau(\sigma))
P(M(T),A(T))+\beta^{2}(\tau(\sigma))[M_{2,1}(T)+M_{3,0}(T)]d\sigma\\
&\lesssim&
\beta^{5/2}(T)P(M(T),A(T))+\beta^{2}(T)[M_{3,0}(T)+M_{2,1}(T)];
\end{array}
\end{equation} and the slow decay of $\beta$
\begin{equation}\label{eq:Z1}
Z_{1}\lesssim \exp{-c_{0}S}\|\phi(\cdot,0)\|_{1,2}\lesssim
\beta^{2}(T)M_{1,2}(0).
\end{equation} Estimates (~\ref{eq:Eta12})-(~\ref{eq:Z1}) yield the bound $$
\begin{array}{lll}
\beta^{-2}(T)\|\exp{\frac{\alpha
z^{2}}{4}}P^{\alpha}_{1}(\partial_{z}+\frac{\alpha}{2}z)^{2}\eta(S)\|_{1,0}
 \lesssim
M_{1,2}(0)+M_{3,0}(T)+M_{2,1}(T)+\beta^{ \frac{1}{2}
}(0)P(M(T),A(T)).
\end{array}
$$

By Equation (~\ref{eq:AgreeEnd}) we obtain $\|\exp{\frac{\alpha
z^{2}}{4}}P^{\alpha}_{1}(\partial_{z}+\frac{\alpha}{2}z)^{2}\eta(S)\|_{1,0}=\|\partial_{y}^{2}\phi(T)\|_{1,0}$
which together with the definition of $M_{1,2}$ yields
$$M_{1,2}(T)\lesssim
M_{1,2}(0)+M_{3,0}(T)+M_{2,1}(T)+\beta^{ \frac{1}{2}
}(0)P(M(T),A(T)).$$ Since $T$ is arbitrary (~\ref{eq:M12}) follows.

\begin{thebibliography}{10}

\bibitem{MR2029906}
N.~D. Alikakos and A.~Freire.
\newblock The normalized mean curvature flow for a small bubble in a
  {R}iemannian manifold.
\newblock {\em J. Differential Geom.}, 64(2):247--303, 2003.

\bibitem{AlAnGi}
S.~Altschuler, S.~B. Angenent, and Y.~Giga.
\newblock Mean curvature flow through singularities for surfaces of rotation.
\newblock {\em J. Geom. Anal.}, 5(3):293--358, 1995.

\bibitem{SK2}
S.~Angenent and D.~Knopf.
\newblock Precise asymptotics of the {R}icci flow neckpinch.
\newblock {\em arXiv:math.DG/0511247, v1}.

\bibitem{MR2092903}
S.~Angenent and D.~Knopf.
\newblock An example of neckpinching for {R}icci flow on {$S\sp {n+1}$}.
\newblock {\em Math. Res. Lett.}, 11(4):493--518, 2004.

\bibitem{MR1427656}
S.~B. Angenent and J.~J.~L. Vel{\'a}zquez.
\newblock Degenerate neckpinches in mean curvature flow.
\newblock {\em J. Reine Angew. Math.}, 482:15--66, 1997.

\bibitem{MR1456315}
M.~Athanassenas.
\newblock Volume-preserving mean curvature flow of rotationally symmetric
  surfaces.
\newblock {\em Comment. Math. Helv.}, 72(1):52--66, 1997.

\bibitem{MR1979113}
M.~Athanassenas.
\newblock Behaviour of singularities of the rotationally symmetric,
  volume-preserving mean curvature flow.
\newblock {\em Calc. Var. Partial Differential Equations}, 17(1):1--16, 2003.

\bibitem{MR485012}
K.~A. Brakke.
\newblock {\em The motion of a surface by its mean curvature}, volume~20 of
  {\em Mathematical Notes}.
\newblock Princeton University Press, Princeton, N.J., 1978.

\bibitem{BrKu}
J.~Bricmont and A.~Kupiainen.
\newblock Universality in blow-up for nonlinear heat equations.
\newblock {\em Nonlinearity}, 7(2):539--575, 1994.

\bibitem{MR1100211}
Y.~G. Chen, Y.~Giga, and S.~Goto.
\newblock Uniqueness and existence of viscosity solutions of generalized mean
  curvature flow equations.
\newblock {\em J. Differential Geom.}, 33(3):749--786, 1991.

\bibitem{DGSW}
S.~Dejak, Z.~Gang, I.~M. Sigal, and S.~Wang.
\newblock Blowup of nonlinear heat equations.
\newblock {\em to appear in Advances in Applied Mathematics}.

\bibitem{MR1122308}
G.~Dziuk and B.~Kawohl.
\newblock On rotationally symmetric mean curvature flow.
\newblock {\em J. Differential Equations}, 93(1):142--149, 1991.

\bibitem{MR1139032}
K.~Ecker.
\newblock Local techniques for mean curvature flow.
\newblock In {\em Workshop on Theoretical and Numerical Aspects of Geometric
  Variational Problems (Canberra, 1990)}, volume~26 of {\em Proc. Centre Math.
  Appl. Austral. Nat. Univ.}, pages 107--119. Austral. Nat. Univ., Canberra,
  1991.

\bibitem{MR1770903}
L.~C. Evans and J.~Spruck.
\newblock Motion of level sets by mean curvature. {I}.
\newblock {\em J. Differential Geom.}, 33(3):635--681, 1991.

\bibitem{MR840401}
M.~Gage and R.~S. Hamilton.
\newblock The heat equation shrinking convex plane curves.
\newblock {\em J. Differential Geom.}, 23(1):69--96, 1986.

\bibitem{GaSi}
Z.~Gang.
\newblock The quenching problem of nonlinear heat equations.
\newblock {\em arXiv:math.AP/0612110}.

\bibitem{GK1}
Y.~Giga and R.~V. Kohn.
\newblock Asymptotically self-similar blow-up of semilinear heat equations.
\newblock {\em Comm. Pure Appl. Math.}, 38(3):297--319, 1985.

\bibitem{MR1016434}
M.~A. Grayson.
\newblock A short note on the evolution of a surface by its mean curvature.
\newblock {\em Duke Math. J.}, 58(3):555--558, 1989.

\bibitem{MR664497}
R.~S. Hamilton.
\newblock Three-manifolds with positive {R}icci curvature.
\newblock {\em J. Differential Geom.}, 17(2):255--306, 1982.

\bibitem{Huis1}
G.~Huisken.
\newblock Flow by mean curvature of convex surfaces into spheres.
\newblock {\em J. Differential Geom.}, 20(1):237--266, 1984.

\bibitem{MR837523}
G.~Huisken.
\newblock Contracting convex hypersurfaces in {R}iemannian manifolds by their
  mean curvature.
\newblock {\em Invent. Math.}, 84(3):463--480, 1986.

\bibitem{Huis2}
G.~Huisken.
\newblock Asymptotic behavior for singularities of the mean curvature flow.
\newblock {\em J. Differential Geom.}, 31(1):285--299, 1990.

\bibitem{MR1216584}
G.~Huisken.
\newblock Local and global behaviour of hypersurfaces moving by mean curvature.
\newblock In {\em Differential geometry: partial differential equations on
  manifolds (Los Angeles, CA, 1990)}, volume~54 of {\em Proc. Sympos. Pure
  Math.}, pages 175--191. Amer. Math. Soc., Providence, RI, 1993.

\bibitem{MR1482035}
G.~Huisken.
\newblock Lecture two: singularities of the mean curvature flow.
\newblock In {\em Tsing Hua lectures on geometry \& analysis (Hsinchu,
  1990--1991)}, pages 125--130. Int. Press, Cambridge, MA, 1997.

\bibitem{MR1666878}
G.~Huisken and C.~Sinestrari.
\newblock Mean curvature flow singularities for mean convex surfaces.
\newblock {\em Calc. Var. Partial Differential Equations}, 8(1):1--14, 1999.

\bibitem{MR1189906}
T.~Ilmanen.
\newblock Generalized flow of sets by mean curvature on a manifold.
\newblock {\em Indiana Univ. Math. J.}, 41(3):671--705, 1992.

\bibitem{LSU}
O.~A. Lady{\v{z}}enskaja, V.~A. Solonnikov, and N.~N.
Ural{\cprime}ceva.
\newblock {\em Linear and quasilinear equations of parabolic type}.
\newblock Translated from the Russian by S. Smith. Translations of Mathematical
  Monographs, Vol. 23. American Mathematical Society, Providence, R.I., 1967.

\bibitem{SM}
M.~Simon.
\newblock Mean curvature flow of rotationally symmetric hypersurfaces.
\newblock {\em Thesis}, 1990.

\bibitem{SMO}
K.~Smoczyk.
\newblock The evolution of special hypersurfaces by their mean curvature.
\newblock {\em Preprint, Ruhr-University Bochum 1993}.

\bibitem{SS}
H.~M. Soner and P.~E. Souganidis.
\newblock Singularities and uniqueness of cylindrically symmetric surfaces
  moving by mean curvature.
\newblock {\em Comm. Partial Differential Equations}, 18(5-6):859--894, 1993.

\bibitem{WMTao}
M.-T. Wang.
\newblock Mean curvature flow in higher codimension.
\newblock {\em arXiv:math/0204054v1}, 2002.

\end{thebibliography}
\def\cprime{$'$} \def\cprime{$'$} \def\cprime{$'$} \def\cprime{$'$}
  \def\cprime{$'$} \def\cprime{$'$}

\end{document}